\documentclass{amsart}
\usepackage{graphicx}
\usepackage{amssymb,latexsym,wasysym}
\usepackage{amsfonts,amsmath,amsthm}
\usepackage{color}
\usepackage{enumerate}
\usepackage[dvipsnames]{xcolor}
\usepackage{url}
\usepackage[colorlinks=true,
            linkcolor=violet,
            urlcolor=Aquamarine,
            citecolor=YellowOrange]{hyperref}
\def\({\left (}
\def\){\right )}
\def\<{\left\langle}
\def\>{\right\rangle}

\newtheorem{thm}{Theorem}[section]
\newtheorem{cor}[thm]{Corollary}
\newtheorem{lem}[thm]{Lemma}

\newtheorem{rem}[thm]{Remark}

\newcommand{\norm}[1]{\left\Vert#1\right\Vert}
\newcommand{\abs}[1]{\left\vert#1\right\vert}
\newcommand{\set}[1]{\left\{#1\right\}}
\newcommand{\Real}{\mathbb R}

\newcommand{\bx}{\mathrm{x}}

\numberwithin{equation}{section}

\newcommand{\red}[1]{ {\color{red} #1} }

\begin{document}
\title{On the blow-up of Yang-Mills fields in dimension four}

\begin{abstract}
	In this paper, we study the blow-up of a sequence of Yang-Mills connection with bounded energy on a four manifold. We prove a set of equations relating the geometry of the bubble connection at the infinity with the geometry of the limit connection at the energy concentration point. These equations exclude certain scenarios from happening, for example, there is no sequence of Yang-Mills $SU(2)$ connections on $S^4$ converging to an ASD one-instanton while developing a SD one-instanton as a bubble. The proof involves the expansion of connection forms with respect to some Coulomb gauge on long cylinders.
\end{abstract}
\author{Hao Yin}

\address{Hao Yin,  School of Mathematical Sciences,
University of Science and Technology of China, Hefei, China}
\email{haoyin@ustc.edu.cn }
\thanks{The research of Hao Yin is supported by NSFC 11971451 and 2020YFA0713102.}
\maketitle

\section{Introduction}
Suppose that $M$ is a closed oriented four manifold with a Riemannian metric $g$ and $E$ is a smooth vector bundle with prescribed bundle metric. Given a metric-compatible connection $D$, the Yang-Mills functional is defined by
\begin{equation}
	\mathcal Y \mathcal M(D)= \int_M \abs{F_D}^2 dV_g.
	\label{eqn:ym}
\end{equation}
Critical points of the Yang-Mills functional are called Yang-Mills connections. Recall that self-dual(SD) and anti-self-dual(ASD) connections are minimizers and therefore are special Yang-Mills connections. 

Given a sequence of Yang-Mills connections $D_k$ with bounded energy, the compactness was studied by Uhlenbeck in \cite{uhlenbeck1982connections}. For simplicity, we assume that $\bx_0\in M$ is the only energy concentration point and that there is only one bubble. That is, (by passing to subsequence if necessary) $D_k$ up to gauge transformation converges locally smoothly to another Yang-Mills connection $D_\infty$ over $M\setminus \set{\bx_0}$. By the removable singularity theorem \cite{uhlenbeck1982removable}, $D_\infty$ is a smooth Yang-Mills connection of another bundle $E_\infty$ over $M$. Meanwhile, there is a sequence $\lambda_k\to 0$ such that after scaling by $\lambda_k$, the sequence converges (up to gauge transformation) to a connection $D_b$ of the trivial bundle over $\Real^4$. By the removable singularity theorem again, we may assume that $D_b$ is a Yang-Mills connection of the bundle $E_b$ over $S^4=\Real^4\cup \set{\infty}$. 
It is natural to ask:

{\bf Question.} Is there any relation between the bubble connection $D_b$ and the limit connection $D_\infty$?

The following question is closely related and more precise.

{\bf Question.} Given any two Yang-Mills connections $D_1$ and $D_2$ over $S^4$, can we construct a sequence of Yang-Mills connections that converge to $D_1$ away from an energy concentration point and after scaling, develop a limit bubble $D_2$ on $\Real^4$ (which is identified with $S^4$)?

For a particular type of Yang-Mills connections, namely, the anti-self-dual(ASD), or the self-dual(SD) connections, Taubes \cite{taubes1982self,taubes1984self} developed a gluing method that gives an affirmative answer to the second question if both $D_1$ and $D_2$ are ASD, or SD. Indeed, we can even introduce an extra parameter (an isomorphism between fibers) and prescribe the direction along which the two connections are glued(see Chapter 7.2 of \cite{donaldson1990}).

The main result of this paper is the claim that the curvature tensor $F_{D_\infty}$ at $\bx_0$ and the curvature tensor $F_{D_b}$ at $\infty$ should satisfy some restrictions that will be made clear in the statement of Theorem \ref{thm:main}. As a consequence, we see the answer to the second question above is negative in general. Hence, there does exist some relation between $D_b$ and $D_\infty$.

The idea that these curvature tensors can be compared suggests that there is a natural identification of the two vector spaces in which they reside, i.e.
\[
\left[\Lambda^2(M)\otimes {\rm Hom}(E_\infty)\right]_{\bx_0} \quad \text{and} \quad \left[\Lambda^2(S^4)\otimes {\rm Hom}(E_b)\right]_\infty.
\]
Here $\Lambda^2(M)$ and $\Lambda^2(S^4)$ are the spaces of two forms. While the process of scaling and taking limit identifies $T_{\bx_0} M$ naturally with $\Real^4$, or $T_\infty S^4$, the identification of $(E_\infty)_{\bx_0}$ and  $(E_b)_{\infty}$ is less obvious.

It was proved in \cite{rade1993decay} and \cite{groisser1997sharp} that for small $\delta$ and $x\in B_\delta\setminus B_{\lambda \delta^{-1}}$, we have
\begin{equation}
	\abs{F_{D_k}}(x)\leq C \left( 1 + \frac{r^2}{\abs{x}^4} \right).
	\label{eqn:rade}
\end{equation}
Here we have taken some fixed coordinate system $x$ around $\bx_0$ and $B_r$ stands for the ball of radius $r$ centered at the origin.  
In terms of the cylinder coordinates $(t,\omega)= (\log \abs{x},\frac{x}{\abs{x}})$, \eqref{eqn:rade} implies that the curvature $F_{D_k}$ measured in the conformal metric $\abs{x}^{-2} g$ satisfies  
\begin{equation}
	\abs{F_{D_k}}(t,\omega)\leq C (e^{2t} + \lambda_k^2 e^{-2t}) \quad \forall t\in [\log \lambda_k -\log \delta, \log\delta].	
	\label{eqn:radecylinder}
\end{equation}

The identification between $(E_\infty)_{\bx_0}$ and $(E_b)_\infty$ follows from \eqref{eqn:rade} and some routine discussion about the compactness of Yang-Mills fields, which we summarized in the following theorem.
\begin{thm}
	\label{thm:known} Let $(M,g)$, $E$, $D_k$ be as above. Then we have

	(1) For some small $\delta>0$ and $k$ sufficiently large, there is a trivialization $e_k$ of $E$ over $\Omega_k$ such that 
	\begin{equation}
		\abs{A_k}(t,\omega) \leq C (e^{2t}+ \lambda_k^2 e^{-2t}),
		\label{eqn:aistart}
	\end{equation}
	where $A_k$ is the connection form of $D_k$ with respect to $e_k$.

	(2) A choice of such trivializations (for all $k$) leads to a trivialization $e_\infty$ of $E_\infty$ over $B_\delta$ and a trivilization $e_b$ of $E_b$ over $\Real^4\setminus B_{\delta^{-1}}$.

	(3) Different choices of $e_k$ in (1), as long as it satisfies \eqref{eqn:aistart}, will result in different $e_\infty$ and $e_b$. However, the linear map that identifies $(e_\infty)_{\bx_0}$ with $(e_b)_\infty$ is independent of the choices of $e_k$.
\end{thm}

\begin{rem}
	For Theorem \ref{thm:known}, we do not need the sharp decay estimate in \eqref{eqn:rade}. 
\end{rem}

Given the above theorem, if we may compute $F_{D_\infty}$ and $F_{D_b}$ in the local frame $e_\infty$ and $e_b$ respectively (as given in (2) above), it makes sense to compare directly their components. Normally, when we fix a frame of the bundle and a local coordinate system, the curvature $F$ is written as
\[
	F= \frac{1}{2}\sum_{i,j=1}^4 F_{ij} dx_i\wedge dx_j
\]
where $F_{ij}$'s are matrices and $F_{ij}=-F_{ji}$. For our purpose, we set self-dual(SD) forms
\begin{equation}
	\begin{split}
	\Phi_{+,1}&= dx_1\wedge dx_2 + dx_3\wedge dx_4 \\
	\Phi_{+,2}&= dx_1\wedge dx_3 - dx_2\wedge dx_4 \\
	\Phi_{+,3}&= dx_1\wedge dx_4 + dx_2\wedge dx_3 
	\end{split}
	\label{eqn:Phiplus}
\end{equation}
and anti-self-dual(ASD) forms 
\begin{equation}
	\begin{split}
	\Phi_{-,1}&= dx_1\wedge dx_2 - dx_3\wedge dx_4 \\
	\Phi_{-,2}&= dx_1\wedge dx_3 + dx_2\wedge dx_4 \\
	\Phi_{-,3}&= dx_1\wedge dx_4 - dx_2\wedge dx_3.
	\end{split}
	\label{eqn:Phiminus}
\end{equation}
With the above basis of two forms, $F_{D_\infty}$ is
\begin{equation}
	F_{D_\infty}(0)= \sum_{i=1}^3 F^L_{\pm,i} \Phi_{\pm,i}
	\label{eqn:fd}
\end{equation}
for some constant matrices $F^L_{\pm,i}$.
\begin{rem}
	The superscript $L$ means left. In this paper, the left end of the neck cylinder is attached to the weak limit and the right end to the bubble. 
\end{rem}
For the bubble connection defined on $\Real^4 \setminus B_{\delta^{-1}}$, we use another coordinate system
\[
	(y_1,y_2,y_3,y_4)=\frac{1}{\abs{x}^2}(x_1,x_2,x_3,x_4)
\]
and $\Phi_{\pm,i}(y)$ denotes the two forms in \eqref{eqn:Phiplus} and \eqref{eqn:Phiminus} with $x_i$ replaced by $y_i$. Note that due to the problem of orientation, $\Phi_{+,i}(y)$ is ASD, and $\Phi_{-,i}(y)$ is SD. Similar to \eqref{eqn:fd}, the curvature of $F_b$ is
\begin{equation}
	F_b|_{y=0}=\sum_{i=1}^3 F^R_{\mp,i} \Phi_{\pm,i}(y)
	\label{eqn:fb}
\end{equation}
for some constant matrices $F^R_{\mp,i}$.

The main result of this paper is
\begin{thm}
	\label{thm:main}
	Let $D_k$ be a sequence of Yang-Mills connection as above. Assume that $g$ is conformally flat near the only blow-up point $\bx_0$. With the identification given in Thereom \ref{thm:known}, for the curvature of the weak limit $F_{D_\infty}$ in \eqref{eqn:fd} and the curvature of the bubble connection $F_b$ in \eqref{eqn:fb}, we have
	\begin{equation}
		\langle F^L_{\pm,i}, F^R_{\mp,j} \rangle - \langle F^L_{\pm,j}, F^R_{\mp,i} \rangle = 0	
		\label{eqn:main16}
	\end{equation}
	for any pair $(i,j)=(2,3)$ or $(3,1)$ or $(1,2)$
	and
	\begin{equation}
		\sum_{i=1}^3 \langle F^L_{+,i}, F^R_{-,i} \rangle + \langle F^L_{-,i}, F^R_{+,i} \rangle =0.
		\label{eqn:main7}
	\end{equation}
	Here the inner product of two matrices $A,B$ is given by ${\rm Tr}(A B^t)$.
\end{thm}

\begin{rem}
	A first observation is that if both the bubble connection and the weak limit connection are ASD	(or SD), then the equations \eqref{eqn:main16} and \eqref{eqn:main7} hold trivially. In that sense, our theorem says nothing for the blow-up of instantons. This explains that why the gluing of ASD has no local obstructions, i.e. the gluing can happen anywhere on the manifold and the bubble can be attached at any direction.

	However, for general Yang-Mills connections, the above equations become a set of necessary conditions if the gluing method were to work.
\end{rem}

\begin{rem}
	A natural question is what happens if the equations do not hold. The author\cite{yin2021} exploited such a condition in a variational setting. In that sense, one may want to regard \eqref{eqn:main16} and \eqref{eqn:main7} as an infinitesimal balancing condition.
\end{rem}

The proof of this theorem consists of three parts. The first part, Section \ref{sec:gauge}, is a good choice of gauge for $D_k$ in $\Omega_k$. This gauge has the following three advantages:

(1) it is Coulomb gauge, i.e. $d^*A_k$=0 with $D_k=d+A_k$;

(2) $A_k$ satisfies some normalization conditions (see \eqref{eqn:gg1} and \eqref{eqn:gg2}). These equations help us to establish precise relations between the curvature and the derivatives of connection, which we use to prove \eqref{eqn:main16} and \eqref{eqn:main7};

(3) $A_k$ satisfies some weighted decay estimate of order smaller than $2$. 

The second part (Section \ref{sec:asymptotic}) of the proof is some uniform elliptic regularity over the long cylinder. By using the elliptic system satisfied by $A_k$, we improve the estimate of $A_k$ to an order larger than $2$. See \eqref{eqn:Ak} in Section \ref{sub:thelimit}.

The last part (Section \ref{sec:proof}) shows that the expansion coefficients of $A_k$ are restricted by some equations derived from the divergence free property of the stress-energy tensor
\begin{equation}
	\label{eqn:stress}
	S= \left( g^{ml}\langle F_{im}, F_{jl} \rangle -\frac{1}{4}\abs{F}^2 g_{ij}\right) dx_i\otimes dx_j.
\end{equation}
It is well known that as long as $D$ is Yang-Mills, $\mbox{div}(S)=0$. Moreover, the symmetric two tensor $S$ is trace free in dimension four. Hence, if $X$ is a conformal Killing field, then
\[
	\mbox{div} (S\# X) =0,
\]
which by the divergence theorem implies that
\begin{equation}
	\label{eqn:omegac}
	\int_{\Omega_C} (S\# X\# \partial_t) d\omega=0,
\end{equation}
where $\#$ means contraction, $\Omega_C$ is the center of the neck, i.e. $\set{\frac{1}{2}\log \lambda}\times S^3$ and $d\omega$ is the volume form of round $S^3$. This explains the assumption of being conformally flat in Theorem \ref{thm:main}. If otherwise, we may not have any conformal Killing fields.

If we plug the expanion of $A_k$ that is proved in the second part into \eqref{eqn:omegac} and take the limit $k\to \infty$, we get an equation relating the limit connections $D_\infty$ and $D_b$. We may take $X$ to be either any vector field that generates $SO(4)$, or the radial vector field that generates the scaling. Hence, we get a total of seven equations in Theorem \ref{thm:main}.

To conclude the introduction, we make the following observations. For $SU(2)$ connections, there is explicit formula for the curvature of ASD one-instanton (see Chapter 3.4 of \cite{donaldson1990})
\[
	F= 2 \left( \frac{1}{1+\abs{x}^2} \right)^2 \left( \Phi_{-,1} \mathbf i + \Phi_{-,2} \mathbf j + \Phi_{-,3} \mathbf k \right).
\]
Here $\mathbf i$, $\mathbf j$ and $\mathbf k$ forms an orthonormal basis of ${\mathfrak s\mathfrak u}(2)$. As a consequence, if $D_\infty $ is ASD one-instanon, then, regardless of the point, $(F^L_{-,i})_{i=1,2,3}$ is (up to a constant) an orthonormal basis of ${\mathfrak s\mathfrak u}(2)$ and $F^L_{+,i}=0$. Similarly, if the bubble connection $D_b$ is SD one-instanton, we know $(F^R_{+,i})_{i=1,2,3}$ is (up to a constant) an orthonormal basis of ${\mathfrak s\mathfrak u}(2)$ and $F^R_{-,i}=0$. 
In summary, the matrix
\[
	\langle F^L_{-,i}, F^R_{+,j} \rangle
\]
is (up to a constant) an orthogonal matrix. However, \eqref{eqn:main16} implies that it is symmetric and \eqref{eqn:main7} implies that it is traceless. Since $\dim {\mathfrak s\mathfrak u}(2)=3$, there is no such matrix. This proves the claim that there is no sequence of Yang-Mills $SU(2)$ connections on $S^4$ converging to an ASD one-instanton while developing an SD one-instanton as a bubble.

\section{Preliminaries}
In this section, we list a few basic results whose proofs are either known or simple. We also take this opportunity to fix some notations.

\subsection{The natural identification}
\label{sub:natural}
In this subsection, we give a proof of Theorem \ref{thm:known}. Essentially, the arguments below are part of the theory of compactness of Yang-Mills connections established by Uhlenbeck.

We first introduce some notations. Set
\begin{eqnarray*}
	\Omega_k&=&  [\log \lambda_k -\log \delta,\log \delta]\times S^3\\
\Omega_\infty&=& (-\infty,\log\delta] \times S^3 \\
	\Omega_b &=& [-\log \delta, \infty)\times S^3.
\end{eqnarray*}
We define $\eta_k: \Omega_k\to \Real$ by
\[
\eta_k(t):= e^t + e^{-(t-\log \lambda_k)}= e^t + \lambda_k e^{-t}.
\]
We use $\eta_k$ to measure the decay rate of certain quantities defined on $\Omega_k$. Notice that throughout the paper, we omit the subscript $k$ of $\eta_k$ for simplicity.

As in the introduction, $E|_{B_\delta\setminus B_{\lambda_k \delta^{-1}}}$ is regarded as a smooth bundle over $\Omega_k$ and \eqref{eqn:radecylinder} becomes
\begin{equation}
	\abs{F_{D_k}} \leq C \eta(t)^2.
	\label{eqn:newrade}
\end{equation}

Consider the trivial bundle over $\Omega_\infty$ and fix a global frame $e_{\infty}$. If we identify $\Omega_\infty$ with $B_\delta\setminus \set{0}$, this frame extends trivially to be a global trivialization of $E_\infty$ on $B_\delta$. Similarly, we consider a trivial bundle over $\Omega_b$ and fix a global frame $e_{b}$. It gives a trivialization of $E_b$ over $\Real^4\setminus B_{\delta^{-1}}$, that extends to a neighborhood of $\infty$.

\begin{lem}
	\label{lem:broken}
	For each $k$, there exists a trivialization $e_k$ of $E$ over $\Omega_k$ such that if $D_k=d+A_k$, then 
\begin{equation}
	\norm{A_k}_{C^{1,\alpha}( [t,t+1]\times S^3)} \leq C \eta(t)^{2}, \quad \forall t\in [\log \lambda_k -\log \delta, \log \delta-1].
	\label{eqn:aidecay}
\end{equation}
Here $\alpha$ is any fixed number in $(0,1)$.
\end{lem}

The proof of this lemma is omitted. One can either smoothify the broken gauge constructed by Uhlenbeck in \cite{uhlenbeck1982removable} or mimic the construction therein to construct one by gluing good trivializations that exist on cylinders of length $2$. The existence of the later follows from Uhlenbeck's gauge fixing theorem in \cite{uhlenbeck1982connections}. 

Set
\begin{eqnarray*}
	L_K&=& [\log \delta-K,\log \delta]\times S^3 \\
	R_K&=& [-\log \delta, -\log \delta +K]\times S^3.
\end{eqnarray*}
The matrix valued one forms $A_k$, when restricted to $L_K$ (for any fixed $K$) converges to $A_\infty$ and gives the limit connection
\[
	D_\infty= d+ A_\infty
\]
w.r.t the frame $e_{\infty}$. Similarly, $A_k(t- \log \lambda_k)$ restricted to $R_K$ converges to $A_{b}$ and defines the bubble connection
\[
	D_b=d+A_b
\]
w.r.t the frame $e_{b}$.

By identifying $e_\infty(0)$ with $e_b(\infty)$, we identify $(E_\infty )_{\bx_0}$ and $(E_b)_{\infty }$. However, in the construction above, for each $D_k$, the choice of $e_{k}$ satisfying \eqref{eqn:aidecay} is not unique. It remains to check that this identification is independent of our choice of $e_{k}$. To see this, assume that there is another $e'_{k}$ in which $D_k=d+A'_k$ and
\begin{equation}
	\norm{A'_k}_{C^{1,\alpha}([t,t+1]\times S^3)}\leq C \eta(t)^{2}.
	\label{eqn:ai2decay}
\end{equation}
Denote by $A'_\infty$ and $A'_b$ the corresponding limit connection forms of $E_\infty$ and $E_b$ respectively.

There is a gauge transformation $s_k$ that maps $e_{k}$ to $e'_{k}$. By \eqref{eqn:aidecay} and \eqref{eqn:ai2decay}, we have
\begin{equation}
	\norm{ds_k}_{C^{1,\alpha}([t,t+1]\times S^3)} \leq C \eta(t)^{2}
	\label{eqn:dsdecay}
\end{equation}
for all $t\in [\log \lambda_k -\log \delta, \log \delta-1]$.
Since $s_k$ is always bounded, by taking subsequence, we have two limits, one is $s_\infty$ defined on $(-\infty,\log\delta]\times S^3$ satisfying
\[
	A'_\infty = s_\infty^{-1} ds_{\infty} + s_\infty^{-1} A_\infty s_\infty;
\]
the other one, $s_b$, defined on $[-\log \delta, \infty)\times S^3$, is the limit of $s_i(t-\log \lambda_k)$, satisfying
\[
	A'_b = s_b^{-1} ds_b + s_b^{-1} A_b s_b.
\]
Moreover, due to \eqref{eqn:dsdecay}, $s_\infty(0)=s_b(\infty)$. This implies that the identification of $(E_\infty)_{\bx_0}$ with $(E_b)_\infty$ is independent of the choice of $e_{k}$.

\subsection{Differential forms on $S^3$}
\label{sub:diff}
Suppose that $x_1,\dots,x_4$ are coordinates of $\Real^4$ and $S^3$ is the unit sphere. 
Let $\triangle_{S^3}$ be the Laplace operator for functions on $S^3$ with the round metric. It is well known that the first nonzero eigenvalue for $-\triangle_{S^3}$ is $3$ and the eigenspace is of dimension $4$ and is spanned by an orthogonal basis\footnote{We do not normalize it. The constant could be ugly and it is not useful.}
\begin{equation*}
	\omega_i:= x_i|_{S^3},\qquad i=1,2,3,4.
\end{equation*}
The next eigenvalue is $8$, it suffices for us to notice that $\sqrt{8}>2$. 

Let $\triangle_{h;S^3}$ be the Hodge Laplace and $\iota$ be the embedding of $S^3$ into $\Real^4$.
We now list the first few eigenvalues and eigenspaces of $\triangle_{h;S^3}$ acting on the space of one forms.
Set
\[
	\psi_i = \iota^*(dx_i).
\]
Then 
\[
	\triangle_{h;S^3} \psi_i = 3 \psi_i
\]
and they span the four dimensional eigenspace of $\triangle_{h;S^3}$ with (the smallest) eigenvalue $3$.

Set
\begin{eqnarray*}
	\phi_{-,1}&=&  \iota^*(x_1 dx_2 -x_2 dx_1 -x_3 dx_4 + x_4 dx_3)\\
	\phi_{-,2}&=&  \iota^*(x_1 dx_3 -x_3 dx_1 -x_4 dx_2 + x_2 dx_4)\\
	\phi_{-,3}&=&  \iota^*(x_1 dx_4 -x_4 dx_1 -x_2 dx_3 + x_3 dx_2)
\end{eqnarray*}
and
\begin{eqnarray*}
	\phi_{+,1}&=&  \iota^*(x_1 dx_2 -x_2 dx_1 +x_3 dx_4 - x_4 dx_3)\\
	\phi_{+,2}&=&  \iota^*(x_1 dx_3 -x_3 dx_1 +x_4 dx_2 - x_2 dx_4)\\
	\phi_{+,3}&=&  \iota^*(x_1 dx_4 -x_4 dx_1 +x_2 dx_3 - x_3 dx_2).
\end{eqnarray*}
They are coclosed and they span the six dimensional eigenspace of $\triangle_{h;S^3}$ with eigenvalue $4$. The next eigenvalue of $\triangle_{h;S^3}$ is again $8$. We refer to \cite{folland1989} for more information about eigenvalues and eigenspaces of the Hodge Laplace on the sphere.

Using the metric of $S^3$, these one forms have their duals, which we list below in due order
\begin{eqnarray*}
	X_{-,1}&=&  (x_1\partial_{x_2}-x_2\partial_{x_1}-x_3\partial_{x_4}+x_4\partial_{x_3})|_{S^3} \\
	X_{-,2}&=&  (x_1\partial_{x_3}-x_3\partial_{x_1}-x_4\partial_{x_2}+x_2\partial_{x_4})|_{S^3} \\
	X_{-,3}&=&  (x_1\partial_{x_4}-x_4\partial_{x_1}-x_2\partial_{x_3}+x_3\partial_{x_2})|_{S^3}
\end{eqnarray*}
and
\begin{eqnarray*}
	X_{+,1}&=&  (x_1\partial_{x_2}-x_2\partial_{x_1}+x_3\partial_{x_4}-x_4\partial_{x_3})|_{S^3} \\
	X_{+,2}&=&  (x_1\partial_{x_3}-x_3\partial_{x_1}+x_4\partial_{x_2}-x_2\partial_{x_4})|_{S^3} \\
	X_{+,3}&=&  (x_1\partial_{x_4}-x_4\partial_{x_1}+x_2\partial_{x_3}-x_3\partial_{x_2})|_{S^3}.
\end{eqnarray*}
We observe that $\set{X_{-,1},X_{-,2},X_{-,3}}$ is an orthonormal frame of the trivial bundle $TS^3$, while $\set{X_{+,1},X_{+,2},X_{+,3}}$ is another. Moreover, these six vector fields generate the Lie algebra of $SO(4)$ that acts on $S^3$ isometrically.

By elementary computation, there is a matrix $T\in SO(3)$ such that
\begin{equation}
	\label{eqn:xy}
	\begin{split}
		&\left( (X_{-,i},X_{+,j}) \right)_{i,j=1,2,3} \\
	&= r^{-2}
	\left(
	\begin{array}[]{ccc}
		x_1^2+x_2^2-x_3^2-x_4^2 & 2(x_1x_4+x_2x_3) & 2(x_2x_4-x_1x_3)\\
		2(x_2x_3-x_1x_4) & x_1^2+x_3^2-x_2^2-x_4^2 & 2(x_1x_2+x_3x_4) \\
		2(x_1x_3+x_2x_4) & 2(x_3x_4-x_1x_2) & x_1^2+x_4^2-x_2^2-x_3^2
	\end{array}
	\right)\\
	&:=-T^t.
	\end{split}
\end{equation}
In particular, we have
\[
	X_{-,i}=  (X_{-,i},X_{+,j})X_{+,j}= -(T)_{ji} X_{+,j}
\]
and
\[
	X_{+,i}=  (X_{-,j},X_{+,i})X_{-,j}= -(T)_{ij} X_{-,j}.
\]
Since $\phi_{\pm;i}$ are the dual basis of $X_{\pm;i}$, we also have
\begin{equation}
	\phi_{-;i}=-(T)_{ji}\phi_{+;j};\quad \phi_{+;i}=-(T)_{ij}\phi_{-;j}.
	\label{eqn:phipm}
\end{equation}
\begin{rem}
	By some abuse of notations, we also regard $\phi_{\pm,i}$ as one forms on the cylinder $[t_1,t_2]\times S^3$. Similarly, $X_{\pm,i}$ are vector fields on $[t_1,t_2]\times S^3$ that are perpendicular to $\partial_t$.
\end{rem}

Next, we consider two forms on cylinders. While $\phi_{\pm,i}$ are the main terms in our asymptotic analysis of connection form, $d(\phi_{\pm,i})$ terms in the curvature are essential for the proof of Theorem \ref{thm:main}.

For $i=1,2,3$, set
\begin{equation}
	\label{eqn:defpq}
\mathcal P_{\pm,i}= e^{-2t}d(e^{2t} \phi_{\pm,i}); \quad \mathcal Q_{\pm,i}=e^{2t} d(e^{-2t} \phi_{\pm,i}).
\end{equation}

We summarize the properties of these two forms in the form of a lemma.
\begin{lem}
	\label{lem:pq}
	(1) $\set{\mathcal P_{-,i}}$ is an orthogonal basis of the ASD two forms on $[t_1,t_2]\times S^3$; so is $\set{\mathcal Q_{+,i}}$.

	(2) $\set{\mathcal P_{+,i}}$ is an orthogonal basis of the SD two forms on $[t_1,t_2]\times S^3$; so is $\set{\mathcal Q_{-,i}}$.

	(3) $\abs{\mathcal P_{\pm,a}}=4$ and $\abs{\mathcal Q_{\pm,a}}=4$.

	(4) we have
\[
	\left( 
	\begin{array}[]{c}
		 \mathcal Q_{+,1} \\
		 \mathcal Q_{+,2} \\
		 \mathcal Q_{+,3} \\
	\end{array}
	\right)= T\cdot
	\left(  
	\begin{array}[]{c}
		 {\mathcal P}_{-,1} \\
		 {\mathcal P}_{-,2} \\
		 {\mathcal P}_{-,3} \\
	\end{array}
	\right)
\]
and
\[
	\left( 
	\begin{array}[]{c}
		 \mathcal Q_{-,1} \\
		 \mathcal Q_{-,2} \\
		 \mathcal Q_{-,3} \\
	\end{array}
	\right)= T^t\cdot
	\left(  
	\begin{array}[]{c}
		 {\mathcal P}_{+,1} \\
		 {\mathcal P}_{+,2} \\
		 {\mathcal P}_{+,3} \\
	\end{array}
	\right)
\]
where $T$ is the matrix defined via \eqref{eqn:xy}.
\end{lem}

\begin{proof}
	First, we observe that it suffices to prove (1)-(3) for $\mathcal P_{\pm,i}$, because by \eqref{eqn:defpq}, the orientation reversing isometry $(t,\omega)\mapsto (-t,\omega)$ takes $\mathcal P_{\pm,i}$ to $\mathcal Q_{\pm,i}$.

	Let $\pi$ be the map from cylinder to $\Real^4\setminus \set{0}$ given by
\[
\pi(t,\omega)= e^t \omega.
\]
We may check that 
\begin{equation}
	\label{eqn:check}
	\frac{1}{2}e^{2t}\mathcal P_{\pm,i} = \pi^* \Phi_{\pm,i}.
\end{equation}
In fact, take $i=1$ for example, $\Phi_{+,1}=\frac{1}{2} d(x_1 dx_2-x_2 dx_1+x_3dx_4- x_4dx_3)$. Since $\phi_{+,1}$ is regarded as a form on cylinder indepdent of $t$, we regard $\iota$ as the map $x\mapsto \frac{x}{\abs{x}}$ and compute
\[
	\phi_{+,1}= \iota^*(x_1 dx_2-x_2 dx_1+x_3dx_4- x_4dx_3)=\frac{1}{\abs{x}^2}(x_1 dx_2-x_2 dx_1+x_3dx_4- x_4dx_3).
\]
Muliplying both sides by $e^{2t}$ and taking differential, we obtain \eqref{eqn:check}.

Recall that $\set{\Phi_{\pm,i}}_{i=1}^3$ is an orthogonal basis of SD/ASD forms on $\Real^4$. (1) and (2) then follows from the fact that $\pi$ is conformal.

For (3), we notice that the norm of $\Phi_{\pm,i}$ is $2$. By the conformality of $\pi$,
\[
\abs{\mathcal P_{\pm,i}}= 2 e^{-2t} \abs{\pi^* \Phi_{\pm,i}} = 4.
\]

For a proof of (4), we also use the pullback of $\pi$ to translate this into an equation for forms defined on $\Real^4$. We need the map $\underline{r}:\Real^4\setminus \set{0}\to \Real^4\setminus \set{0}$ given by
\[
	\underline{r}(x)=\frac{x}{\abs{x}^2}.
\]
It follows from \eqref{eqn:check} that
\begin{equation}
	\label{eqn:check2}
	\frac{1}{2}e^{-2t}\mathcal Q_{\pm,i}= \pi^* \underline{r}^* \Phi_{\pm,i}.
\end{equation}
With \eqref{eqn:check} and \eqref{eqn:check2}, it suffices to prove
\[
	\underline{r}^*
	\left( 
	\begin{array}[]{c}
		 \Phi_{+,1} \\
		 \Phi_{+,2} \\
		 \Phi_{+,3} \\
	\end{array}
	\right)=  \frac{1}{\abs{x^4}} T\cdot
	\left(  
	\begin{array}[]{c}
		 {\Phi}_{-,1} \\
		 {\Phi}_{-,2} \\
		 {\Phi}_{-,3} \\
	\end{array}
	\right),
\]
which we may verify via direct computation. The same computation applies to the other equation.
\end{proof}
Later in our proofs, we need to compute $\iota_V \mathcal W$, where $V$ is one of
\[
	\partial_t, X_{\pm,i}
\]
and $\mathcal W$ is one of
\[
	\mathcal P_{\pm,i}, \mathcal Q_{\pm,i}
\]
for $i=1,2,3$.

We present the results in the form of a table
\begin{equation}
\label{eqn:table}
	\begin{array}[]{|c||c|c|c|c|c|c|}
		\hline
		\frac{1}{2}\iota_V\mathcal W& \mathcal P_{-,1} & \mathcal P_{-,2} & \mathcal P_{-,3} & \mathcal P_{+,1} & \mathcal P_{+,2} & \mathcal P_{+,3}  \\
	 \hline
	 \hline
	 \partial_t  & \phi_{-,1} & \phi_{-,2} & \phi_{-,3}  & \phi_{+,1} &  \phi_{+,2} &  \phi_{+,3}  \\
	 \hline
	 X_{-,1} & -\partial_t &  \phi_{-,3} & - \phi_{-,2} &  &  &   \\
	 \hline
	 X_{-,2} &  - \phi_{-,3} & -\partial_t &  \phi_{-,1} &  &  &   \\
	 \hline
	 X_{-,3} &  \phi_{-,2} & - \phi_{-,1} & -\partial_t &  &  &   \\
	 \hline
	 X_{+,1} &  &  &  & -\partial_t & - \phi_{+,3} &  \phi_{+,2}  \\
	 \hline
	 X_{+,2} &  &  &  &  \phi_{+,3} & -\partial_t & - \phi_{+,1}  \\
	 \hline
	 X_{+,3} &  &  &  & - \phi_{+,2} &  \phi_{+,1} & -\partial_t \\
	 \hline
	\end{array}
\end{equation}
Notice that the blanks in the table can be filled by using the transition matrix in \eqref{eqn:xy}. For example,
\begin{eqnarray*}
	\iota_{X_{-,i}} \mathcal P_{+,j}&=& - (T)_{mi}\iota_{X_{+,m}}  \mathcal P_{+,j}.
\end{eqnarray*}

There should be a similar table as above with all $\mathcal P$ replaced by $\mathcal Q$.
In future applications, we care only about the projection of $\iota_{V}\mathcal W$ along the $TS^3$ direction. In this sense, the following lemma allows us to compute the table for $\mathcal W=\mathcal Q_{\pm,i}$.
\begin{lem}\label{lem:PQ}
	We also have
	\begin{equation*}
		\iota_{X_{\pm;i}} \mathcal P_{\pm,j}= \iota_{X_{\pm;i}} \mathcal Q_{\pm,j} \quad {\rm mod} \quad dt.
	\end{equation*}
\end{lem}
\begin{proof}
The proof is by direct computation. We write $X$ for any vector field tangent to $S^3$.
\begin{eqnarray*}
	\iota_{X} \mathcal P_{\pm,j}&=&  e^{-2t} \iota_X (d (e^{2t}\phi_{\pm,j})) \\
	&=& \iota_X \left(  2dt\wedge \phi_{\pm,j} + d\phi_{\pm,j} \right) \\
	&=& \iota_X \left(  -2dt\wedge \phi_{\pm,j} + d\phi_{\pm,j} \right) \quad \mbox{mod} \quad dt \\
	&=& \iota_X (e^{2t} d(e^{-2t}\phi_{\pm,j}))\quad \mbox{mod} \quad dt \\
	&=& \iota_X \mathcal Q_{\pm,j}\quad \mbox{mod} \quad dt .
\end{eqnarray*}
Here we have used the fact that $\iota_X (dt\wedge \phi_{\pm,j})$ is a multiple of $dt$ as long as $dt(X)=0$.
\end{proof}

\section{Gauge fixing on the neck}
\label{sec:gauge}

Usually when we study a connection near a point, we may take the Coulomb gauge, in which $d^*A=0$. The advantage is that the PDE satisfied by a Yang-Mills connection becomes elliptic in this gauge. There are other situations in which we want to use the normal gauge, namely, a local frame obtained by parallel transportation along radial directions. The advantage of this frame is that the connection form $A$ vanishes at $1+\alpha$ order at the origin for any $\alpha\in (0,1)$. A natural question is whether we can find a local frame that combines these two features. The answer is yes, and it can be proved by using the techniques developed in this section. Since it is not used in this paper, we shall not prove it.

In this section, we carry the above idea further into the discussion of a connection defined on a long cylinder $\Omega$. We prove the existence of a gauge in which, first, the connection form $A$ satisfies $d^*A=0$ (with respect to the round cylinder metric), second, the size of $A$ (measured with respect to cylinder metric as well) decays exponentially fast at an order between $\sqrt{3}$ and $2\sqrt{2}-1$ as we moves from the boundary into the middle of the cylinder. In the next section, we shall work in this gauge and prove better regularity for $A$ when the connection is Yang-Mills.

As in the introduction, $\Omega$ is $[\log \lambda -\log \delta, \log\delta]\times S^3$. Here $\delta$ is a small constant that can be fixed and $\lambda$ is allowed to vary. It is important that throughout this paper, all constants are independent of $\lambda$, while we allow them to depend on $\delta$.

\begin{thm}
	\label{thm:neckgauge} There is $\epsilon_0>0$. Suppose $\Omega=[\log \lambda-\log\delta,\log\delta]\times S^3$ is a long cylinder and $D$ is a connection satisfying
	\begin{equation}
		\label{eqn:neckassume}
		\abs{F_D} + \abs{\nabla_D F_D} + \abs{\nabla^2_D F_D}\leq \epsilon_0(e^t+\lambda e^{-t})^2.
	\end{equation}
	Then there is a trivialization over $\Omega$ in which $D=d+A$ such that
	\begin{equation}
		\int_{\set{\frac{1}{2}\log \lambda}\times S^3} A(\partial_t) = 0;\quad 
		\int_{\set{\frac{1}{2}\log \lambda}\times S^3} A(\partial_t)\cdot \omega_i = 0;
		\label{eqn:gg1}
	\end{equation}
	\begin{equation}
	\int_{\set{\frac{1}{2}\log \lambda}\times S^3} \partial_t (A(\partial_t))\cdot \omega_i = 0,\quad i=1,2,3,4
		\label{eqn:gg2}
	\end{equation}
	and
	\begin{equation}
	d^*A = 0;\qquad \norm{A}_{\mathcal X_2}\leq C\epsilon_0.
		\label{eqn:gg3}
	\end{equation}
\end{thm}
Here $\omega_i$'s are the eigenfunctions defined in Section \ref{sub:diff} and the $\norm{\cdot}_{\mathcal X_2}$ measures the size of $A$ with an exponential decay weight. The precise definition is given below.

\begin{rem}
	The assumption on the derivative of curvature in \eqref{eqn:neckassume} is obviously not optimal. It is a technical problem in the proof of Lemma \ref{lem:ae}. However, this is not a problem when the connection is Yang-Mills. Moreover, the exponent $2$ in the right hand side of \eqref{eqn:neckassume} may be replaced by $\alpha_2<2$ (see below) without any consequence.
\end{rem}

\subsection{Function spaces}
In this subsection, we give definitions of various function spaces that are needed for the statement and the proof of Theorem \ref{thm:neckgauge}.

Let $\Omega$ be the cylinder $[\log \lambda-\log \delta, \log \delta]\times S^3$ in Theorem \ref{thm:neckgauge}. 
Set
\begin{eqnarray*}
	\Omega_t &=&  \set{t}\times S^3 \\
	\Omega_L &=&  \set{\log \delta}\times S^3 \\
	\Omega_R &=&  \set{\log \lambda - \log \delta}\times S^3 \\
	\Omega_C &=&  \set{\frac{1}{2}\log \lambda}\times S^3.
\end{eqnarray*}
Denote the interval $[\log \lambda -\log \delta, \log \delta-1]$ by $I$ and for any $t\in I$, set
\[
	\Omega_{[t]}:=[t,t+1]\times S^3.
\]

Since $\Omega$ is a manifold with boundary, there is a well defined $C^{k,\alpha}(\Omega)$ space of functions with its $C^{k,\alpha}$ norm for $k\in \mathbb N\cup \set{0}$ and $\alpha\in (0,1)$. In this paper, we write $C^{k+\alpha}(\Omega)$ instead of $C^{k,\alpha}(\Omega)$. 
Throughout the paper, we pick and fix $\alpha_1\in (\sqrt{3}+1,\sqrt{8})$ and set $\alpha_2=\alpha_1-1$ and $\alpha_3=\alpha_1-2$. Obviously, $\alpha_2\in (\sqrt{3},2)$ and $\alpha_3\in (0,\sqrt{3})$. 

Recall that on $\Omega$, the usual $C^\alpha$ norm (of functions and one forms) is defind as
\[
	\norm{f}_{C^\alpha}:= \sup_{t\in I} \norm{f}_{C^\alpha(\Omega_{[t]})}.
\]
Denote by $D^\alpha(\Omega)$ be the space of $C^\alpha$ one forms on $\Omega$ with
\[
	\norm{\omega}_{C^\alpha}:= \sup_{t\in I} \norm{\omega}_{C^\alpha(\Omega_{[t]})}.
\]

Define $\mathcal X_1$ to be the subspace of $C^{\alpha_1}(\Omega)$ satisfying
\begin{equation}
	\int_{\Omega_C} u d\omega = 0.
	\label{eqn:x1}
\end{equation}
For a function $u\in \mathcal X_1$, we define 
\[
	\norm{u}_{\mathcal X_1} := \sup_{t\in I} \norm{u}_{C^{\alpha_1}(\Omega_{[t]})} \cdot \eta^{-\alpha_2}(t)
\]
where we recall that $\eta(t)= e^t + e^{-(t-\log \lambda)}$.
Then $(\mathcal X_1, \norm{\cdot}_{\mathcal X_1})$ is a Banach space.
\begin{rem}
	(1) We remark that for a fixed $\lambda$, the norms $\norm{\cdot}_{\mathcal X_1}$ and $\norm{\cdot}_{C^{\alpha_1}}$ are equivalent norms (by a constant depending on $\lambda$). However, since we are interested in estimates that are uniform in $\lambda$, they are different.

	(2) For the purpose of this paper, the functions in $\mathcal X_1$ are regarded as Lie-algebra valued, or matrix-valued. This remark also applies to the spaces to be defined below.
\end{rem}

Recall that $\omega_1,\dots,\omega_4$ are the four eigenfunctions of $\triangle_{S^3}$ corresponding to eigenvalue $-3$. Let $\Psi$ be the orthogonal projection to the complement of the $5$-dimensional subspace spanned by $1,\omega_1,\dots,\omega_4$ in $C^{\alpha_2}(S^3)$. Denote the image of $\Psi$ by $C^{\alpha_2}_{\Psi}$.

Let $\mathcal X_2$ be $D^{\alpha_2}$, with its norm defined by
\[
	\norm{A}_{\mathcal X_2}:= \sup_{t\in I} \norm{A}_{C^{\alpha_2}(\Omega_{[t]})} \cdot \eta^{-\alpha_2}(t).
\]
Similarly, $\mathcal X_3$ is defined to be the space of $C^{\alpha_3}$ functions, but equipped with the norm
\[
	\norm{u}_{\mathcal X_3}:= \sup_{t\in I} \norm{u}_{C^{\alpha_3}(\Omega_{[t]})} \cdot \eta^{-\alpha_2}(t).
\]

Finally, define $\mathcal Y$ to be the product space 
\[
	\mathcal Y:= \mathcal X_3 \times C^{\alpha_2}_{\Psi} \times C^{\alpha_2}_{\Psi}\times \mathfrak g^9
\]
with the norm
\begin{equation}
	\label{eqn:normy}
	\begin{split}
		\norm{(v,v_L,v_R,b_0,a_i,b_i)}_{\mathcal Y}:=&  \norm{v}_{\mathcal X_3}+ \norm{v_L}_{C^{\alpha_2}(\Omega_L)}+ \norm{v_R}_{C^{\alpha_2}(\Omega_R)} \\
							     & + \frac{\abs{b_0}+\sum_{i=1}^4 (\abs{a_i}+\abs{b_i})}{\lambda^{\alpha_2/2}}.
	\end{split}
\end{equation}
Here $\mathfrak g$ is the Lie-algebra, or equivalently, one may take it as the space of skew-symmetric matrics.

We shall be concerned with the map 
\[
	\Phi: \mathcal X_1\times \mathcal X_2 \to \mathcal Y
\]
given by
\begin{equation}
	\begin{split}
		\Phi(u,A):=&  \left( d^*A',\Psi(A'(\partial_t)|_{\Omega_L}), \Psi(A'(\partial_t)|_{\Omega_R}), \int_{\Omega_C}A'(\partial_t),  \right.\\
		& \left.  \int_{\Omega_C} A'(\partial_t)\cdot \omega_i, \int_{\Omega_C} \partial_t(A'(\partial_t))\cdot \omega_i\right)
	\end{split}
	\label{eqn:phi}
\end{equation}
where $A'=e^{-u}de^u + e^{-u}Ae^u$. $\Phi$ is a smooth map.
\footnote{
$\mathcal X_1$, $\mathcal X_2$ and $\mathcal X_3$ are nothing but the usual H\"older space, equipped with some equivalent norms. Hence the smoothness of $\Phi$ is irrelevant to the weighted norms. 
}

\subsection{Continuity method}

Due to a similar argument of Uhlenbeck that justified Lemma \ref{lem:broken}, there is some gauge over $\Omega$ in which
\begin{equation}
	\label{eqn:firsta}
	\norm{A}_{C^{\alpha_1}(\Omega_{[t]})} \leq \epsilon_0 C \eta^{\alpha_2}(t),\quad \forall t\in I.
\end{equation}
Here $C$ is universal and we have used \eqref{eqn:neckassume}, $\alpha_1< 3$ and $\alpha_2<2$.

The proof of Theorem \ref{thm:neckgauge} is by continuity method. We choose the path
\[
	A_\tau:= \tau A
\]
for $t\in [0,1]$. 
By \eqref{eqn:firsta} there is a universal number $C_1$ such that 
\begin{equation}
	\sup_{\tau\in [0,1]} \norm{F_{\tau}}_{\mathcal X_2} \leq C_1 \epsilon_0; \quad \sup_{\tau\in [0,1]} \norm{A_\tau}_{\mathcal X_1}\leq C_1\epsilon_0.
	\label{eqn:ftau}
\end{equation}

Let $E$ be the set of $\tau\in [0,1]$ for which there is a gauge transformation $s=e^u$ (for some $u\in \mathcal X_1$) such that 
\begin{equation}
	\label{eqn:aprime}
	A':= s^{-1}ds +s^{-1}As
\end{equation}
satisfies

(E1) 
\[
d^*A'=0;
\]

(E2) 
\[
\Psi(A'(\partial_t)|_{\Omega_L})=\Psi(A'(\partial_t)|_{\Omega_R})=0;
\]

(E3) for $i=1,2,3,4$
\[
\int_{\Omega_C} A'(\partial_t) = \int_{\Omega_C} A'(\partial_t)\cdot\omega_i =\int_{\Omega_C} \partial_t(A'(\partial_t))\cdot \omega_i=0.
\]

(E4) for some universal constant $C_2$ (determined later in Lemma \ref{lem:ae})
\[
\norm{A'}_{\mathcal X_2}\leq C_2 \epsilon_0.
\]

Obviously $E$ is not empty. We will show that $E$ is both open and closed.

To see that $E$ is closed, assume $\tau_i\in E$ and $\tau_i\to \tau$. Since
\[
	\norm{A_{\tau_i}}_{\mathcal X_2}\leq C_1\epsilon_0 \quad \text{and}  \quad \norm{A'_{\tau_i}}_{\mathcal X_2}\leq C_2\epsilon_0,
\]
we have
\[
	\norm{ds_i}_{C^{\alpha_2}(\Omega_{[t]})}\leq C\epsilon_0 \eta(t)^{\alpha_2}, \qquad \forall t\in I.
\]
Here $s_i=e^{u_i}$ for some $u_i\in \mathcal X_1$. The above inequality implies that
\[
	\norm{du_i}_{C^{\alpha_2}(\Omega_{[t]})}\leq C\epsilon_0 \eta(t)^{\alpha_2}, \qquad \forall t\in I.
\]
Recall that in the definition of $\mathcal X_1$, we assumed that $\int_{\Omega_C} u_id\omega=0$. Hence, we obtain
\[
	\norm{u_i}_{C^{\alpha_1}(\Omega_{[t]})}\leq C\epsilon_0 \eta(t)^{\alpha_2}, \qquad \forall t\in I,
\]
that is, $\norm{u_i}_{\mathcal X_1}\leq C\epsilon_0$.

By passing to a subsequence if necessary, $A'_{\tau_i}$ converges locally in $C^{\alpha_2}$ to $A'_\tau$ and $s_i$ converges locally in $C^{\alpha_1}$ to $s=e^u$. Therefore
\[
	A'_\tau:= s^{-1}ds + s^{-1}A_\tau s
\]
satisfies (E1-4).

To show $E$ is open, assume that $\tau\in E$. Namely, there is $u\in \mathcal X_1$, $s=e^u$ such that $A'_\tau$ given in \eqref{eqn:aprime} satisfies (E1-4). 

{\bf Claim:} there is a neighborhood of $A'_\tau$ in $\mathcal X_2$ norm such that for any $\tilde{A}$ is this neighborhood, there is some $\tilde{u}\in \mathcal X_1$ satisfying
\[
	\Phi(\tilde{u},\tilde{A})=0.
\]
This is equivalent to the requirement that $\tilde{A}'= e^{-\tilde{u}} de^{\tilde{u}} + e^{-\tilde{u}}\tilde{A} e^{\tilde{u}}$ satisfies (E1-3).

For the proof of the claim, it suffices to apply the implicit function theorem at $(0,A'_\tau)$. For that purpose, we compute
\begin{equation*}
	\begin{split}
	D_1\Phi|_{(0,A'_\tau)}v & = \bigg(-\triangle v + d^*(A'_\tau v-vA'_\tau), \\
	& \Psi(\partial_t v-vA'_\tau(\partial_t)+A'_\tau(\partial_t)v)|_{\Omega_L},\\
	& \Psi(\partial_t v-vA'_\tau(\partial_t)+A'_\tau(\partial_t)v)|_{\Omega_R}, \\
	& \int_{\Omega_C} (\partial_t v-vA'_\tau(\partial_t)+A'_\tau(\partial_t)v),\\
	& \int_{\Omega_C} (\partial_t v-vA'_\tau(\partial_t)+A'_\tau(\partial_t)v)\cdot \omega_i,\\
	&  \int_{\Omega_C} \partial_t (\partial_t v-vA'_\tau(\partial_t)+A'_\tau(\partial_t)v)\cdot \omega_i \bigg).
	\end{split}
\end{equation*}
When $A'_\tau=0$, this reduces to 
\begin{equation*}
	\begin{split}
	D_1\Phi|_{(0,0)}v & = \bigg(-\triangle v, \Psi(\partial_t v)|_{\Omega_L}, \Psi(\partial_t v)|_{\Omega_R}, \\
	& \int_{\Omega_C} (\partial_t v),\int_{\Omega_C} (\partial_t v)\cdot \omega_i, \int_{\Omega_C} \partial_t (\partial_t v)\cdot \omega_i \bigg).
	\end{split}
\end{equation*}
Now, the proof of the claim follows from the following two lemmas.
\begin{lem}
	\label{lem:iso}
	$D_1\Phi_{(0,0)}$ is an isomorphism from $\mathcal X_1$ to $\mathcal Y$.
\end{lem}
\begin{lem}
	\label{lem:perturb}
	For some sufficiently small $\epsilon_0$, there exists $C(\epsilon_0)>0$ such that if $\norm{A}_{\mathcal X_2}\leq C_2\epsilon_0$ (for the $C_2$ in (E4)),
	\[
		\norm{D_1\Phi|_{(0,A)}-D_1\Phi|_{(0,0)}}_{\mathcal X_1;\mathcal Y}\leq C(\epsilon_0).
	\]
	Here $\norm{\cdot}_{\mathcal X_1;\mathcal Y}$ is the norm of bounded linear operators.
	Moreover, $C(\epsilon_0)$ can be as small as we need if $\epsilon_0$ is chosen to be small.
\end{lem}

Lemma \ref{lem:iso} and Lemma \ref{lem:perturb} shows that if $\epsilon_0$ is chosen small, then the implicit function theorem for $\Phi$ at $(0,A'_\tau)$ proves the claim. The proof of Lemma \ref{lem:iso} is postponed to the next subsection and the proof of Lemma \ref{lem:perturb} is omitted, because it is elementary and it suffices to notice that the factor $\lambda^{\alpha_2/2}$ in the definition of the norm of $\mathcal Y$ fits well with the decay in the definition of $\norm{\cdot}_{\mathcal X_2}$.

The gauge transformation that takes $A_\tau$ to $A'_\tau$ pulls back this neighborhood of $A'_\tau$ to $A_\tau$. Hence, to finish the proof of the openness of $E$, we still need
\begin{lem}\label{lem:ae}
	There is some universal constant $\epsilon_1>0$.
	Suppose that $A$ is a connection form in $\mathcal X_2$ satisfying (E1-3) and
	\begin{equation}
		\norm{A}_{\mathcal X_2}\leq \epsilon_1.
		\label{eqn:ae1}
	\end{equation}
	If 
	\begin{equation}
		\norm{F_A}_{C^{\alpha_2}(\Omega([t]))}\leq C_1 \epsilon_0\eta^{\alpha_2}(t).
		\label{eqn:fa}
	\end{equation}
	then for some universal constant $C_2$ we have
	\[
		\norm{A}_{\mathcal X_2}\leq C_2\epsilon_0.
	\]
\end{lem}
Notice that this last inequality is (E4) and that the assumption \eqref{eqn:fa} is guaranteed by \eqref{eqn:ftau} along the continuity path. Moreover, the constant $\epsilon_0$ in Theorem \ref{thm:neckgauge} is determined both by requiring $C_2\epsilon_0< \epsilon_1$ and by Lemma \ref{lem:perturb}.

\begin{rem}
	Ideally, given a $C^{\alpha_2}$ bound of curvature, one would expect a control over $C^{\alpha_1}$ norm of $A$. As can be seen from the following proof, this is indeed the case. However, we simply do not need that much.
\end{rem}

\begin{proof}
	Let 
	\[
		A= f(x,t) dt + \xi(x,t)
	\]
	where $f$ is a function and $\xi$ is a one-form satisfying $\xi(\partial_t)\equiv 0$.
	It follows from $d^*A=0$ that
	\begin{equation}
		-\partial_t f + d^*_{S^3} \xi =0.
		\label{eqn:dSA}
	\end{equation}
	By setting $h=F_A- [A\wedge A]$, we have $dA=h$. On the other hand,
	\begin{equation}\label{eqn:dxi}
		dA= d_{S^3}f \wedge dt - \partial_t \xi \wedge dt + d_{S^3} \xi,
	\end{equation}
which implies that
\begin{equation}
	\partial_t \xi - d_{S^3} f =\iota_{\partial_t} h.
		\label{eqn:dAh}
\end{equation}
Summing the $\partial_t$ of \eqref{eqn:dSA} and the $d^*_{S^3}$ of \eqref{eqn:dAh}, we obtain
\[
	\partial_t^2 f + \triangle_{S^3} f =  d^*_{S^3} (\iota_{\partial_t} h).
\]
The boundary condition of $A$ in (E2) and (E3) translates into
\[
	\Psi(f|_{\Omega_R})=\Psi(f|_{\Omega_R})=0
\]
and
\[
	\int_{\Omega_C} f = \int_{\Omega_C} f\cdot \omega_i = \int_{\Omega_C} \partial_t f\cdot \omega_i=0.
\]
Integrating \eqref{eqn:dSA} over $\Omega_C$, we observe that $\int_{\Omega_C} \partial_t f =0$.
By the definition of $h$,
\[
	\norm{d^*_{S^3} (\iota_{\partial_t} h)}_{\mathcal X_3} \leq C \norm{A}_{\mathcal X_2}^2 + C \norm{F_A}_{\mathcal X_2}.
\]
It then follows form Lemma \ref{lem:dirichlet} that
\begin{equation}
	\norm{f}_{\mathcal X_1} \leq C \norm{A}_{\mathcal X_2}^2 + C\norm{F_A}_{\mathcal X_2}.
	\label{eqn:goodf}
\end{equation}

Next, consider the restriction of the two form $dA=h$ to $S^3$ (in \eqref{eqn:dxi}), we obtian $d_{S^3} \xi= h|_{S^3}$. Together with \eqref{eqn:dSA} and \eqref{eqn:goodf}, we have
\[
	\norm{d^*_{S^3} \xi}_{\mathcal X_2} + \norm{d_{S^3}\xi}_{\mathcal X_2} \leq C\norm{A}_{\mathcal X_2}^2 +C\norm{F_A}_{\mathcal X_2}.
\]
Since there is no harmonic one forms on $S^3$, the elliptic estimate implies that $\norm{\xi}_{\mathcal X_1}\leq C\norm{A}_{\mathcal X_2}^2+ C\norm{F_A}_{\mathcal X_2}$. (The $\partial_t \xi$ part is estimated via \eqref{eqn:dAh} and \eqref{eqn:goodf}.) Together with \eqref{eqn:goodf}, we obtain
\[
	\norm{A}_{\mathcal X_1} \leq C \norm{A}_{\mathcal X_2}^2 + C \norm{F_A}_{\mathcal X_2}.
\]
The proof is done by taking $\epsilon_1$ small.
\end{proof}

\subsection{The linear isomorphism}\label{sub:iso}
In this section, we provide the proof of Lemma \ref{lem:iso}. The essential part of the proof is an elliptic estimate for Poisson equation on the long cylinder, whose proof can be found in the appendix.

\begin{proof}[Proof of Lemma \ref{lem:iso}]
	For simplicity, we write $\mathcal T$ for $D_1\Phi|_{(0,0)}$.
	It is trivial to check that $\mathcal T$ maps a function in $\mathcal X_1$ to $\mathcal Y$ and that $\mathcal T$ is a bounded linear map. To see this, we shall use the fact that $\Psi$ is a bounded linear map from $C^{\alpha_2}(S^3)$ to itself.

	{\bf Step 1.} $\mathcal T$ is a surjective. Given any fixed $(v,v_L,v_R,b_0,a_i,b_i)\in \mathcal Y$, we apply the solvability of boundary value problem of Neumann boundary condition on manifold with boundary. 
Precisely, we solve 
	\[
		\left\{
			\begin{array}[]{l}
				- \triangle u =v \\
				\partial_t u |_{\Omega_L}=v_L - \frac{1}{\abs{S^3}} \int_{\Omega} v (dtd\omega)\\
				\partial_t u |_{\Omega_R}=v_R \\
			\end{array}
			\right.
	\]
	for a function $u$ satisfying $\int_{\Omega_C} u =0$ (as required in the definition of $\mathcal X_1$).
	The boundary conidition on $\Omega_L$ is modified so that the following compatibility condition holds
	\[
		\int_\Omega - v (dtd\omega)= \int_{\Omega_L} \partial_t ud\omega - \int_{\Omega_R} \partial_t ud\omega.
	\]
	In verifying the above equation, we have used the fact that by the definition of $C^{\alpha_2}_\Psi(S^3)$,
	\[
		\int_{\Omega_L} v_L d\omega = \int_{\Omega_R}v_R d\omega =0.
	\]
	The point is that $\Psi(\partial_t u|_{\Omega_L})=v_L$ and $\Psi(\partial_t u|_{\Omega_R})=v_R$ are as required. 

	Next, we modify $u$ so that the rest $9$ conditions hold. Notice that this can be achieved while keeping $\int_{\Omega_C}u=0$, $\Psi(\partial_t u|_{\Omega_L})=v_L$ and $\Psi(\partial_t u|_{\Omega_R})=v_R$. To see this, it suffcies to add a linear combination of
	\begin{equation}
		(t-\frac{1}{2}\log \lambda), e^{\sqrt{3}t}\omega_i, e^{-\sqrt{3}(t-\log \lambda)}\omega_i.
		\label{eqn:combine}
	\end{equation}

	{\bf Step 2.} For any $u\in \mathcal X_1$, we claim that
	\[
		\norm{u}_{\mathcal X_1} \leq C \norm{\mathcal T(u)}_{\mathcal Y}.
	\]
	It follows from this inequality that $\mathcal T$ is injective and its inverse is a bounded linear operator.

For the proof of the claim, if $\mathcal T(u)=(v,v_L,v_R,b_0,a_i,b_i)$, we need to show
\begin{equation}
	\norm{u}_{\mathcal X_1}\leq C \left( \norm{v}_{\mathcal X_3}+ \norm{v_L}_{C^{\alpha_2}(\Omega_L)}+ \norm{v_R}_{C^{\alpha_2}(\Omega_R)} + \frac{\abs{b_0} + \sum_{i=1}^4 (\abs{a_i}+\abs{b_i})}{\lambda^{\alpha_2/2}} \right).
	\label{eqn:step3}
\end{equation}
This is exactly Lemma \ref{lem:neumann}.
\end{proof}

\section{Asymptotic Analysis of the connection}
\label{sec:asymptotic}
Let $D_k$ and $\Omega_k$ be as in the introduction. Since $D_k$ is Yang-Mills, the decay of curvature given in \eqref{eqn:rade} allows us to apply Theorem \ref{thm:neckgauge} to get a Coulomb gauge over $\Omega_k$. The first aim of this section is to use the Yang-Mills equation to improve our understanding of $A_k$. More precisely, we prove $A_k$ has an expansion of the form \eqref{eqn:improveA}, in which the bound of the error term is independent of $k$. 

In the second part, we consider the limit $k\to \infty$ and show how the coefficients in the expansion of $A_k$ on the neck $\Omega_k$ determine the curvature of $D_\infty$ and $D_b$. This is used in the next section to prove Theorem \ref{thm:main}.

\subsection{Expansion of $A$}
In this subsection, for simplicity, we surpress the subscript $k$. 
By Theorem \ref{thm:neckgauge}, we have a global trivialization in which $D=d+A$ satisfying \eqref{eqn:gg1}-\eqref{eqn:gg3}.
The Yang-Mills equation becomes
\begin{equation}
	\triangle_h A + dA\# A+ A\#A\#A= 0 \qquad \text{on} \qquad \Omega.
	\label{eqn:localym}
\end{equation}
Here $\triangle_h$ is the Hodge Laplace for one forms on the cylinder and we have used $d^*A=0$.

By the definition of $\norm{\cdot}_{\mathcal X_2}$, we have
\begin{equation}
	\label{eqn:smallA}
	\norm{dA}_{C^{\alpha_3}(\Omega_{[t]})},\norm{A}_{C^{\alpha_2}(\Omega_{[t]})} \leq C\epsilon_0 \eta^{\alpha_2}, \quad \forall t\in I,
\end{equation}
which implies that 
\begin{equation}
	\norm{\triangle_h A}_{C^{\alpha_3}(\Omega_{[t]})}\leq C\epsilon_0 \eta^{\alpha_1}(t), \quad \forall t\in I.
	\label{eqn:hodgeA}
\end{equation}
Here we have used $2\alpha_2>\alpha_1$.

The main result of this section is
\begin{lem}
	\label{lem:improve} Suppose that $A$ is a one-form satisfying \eqref{eqn:smallA} and \eqref{eqn:hodgeA}.
	Then there is another one form $\tilde{A}$ such that

	(1) 
\[
\triangle_h \tilde{A}=\triangle_h A
\]
	and
	\begin{equation}
		\norm{\tilde{A}}_{C^{\alpha_1}(\Omega_{[t]})}\leq C\epsilon_0 \eta^{\alpha_1}(t), \quad \forall t\in [\log \lambda -\log\delta+2, \log\delta-2];
		\label{eqn:another}
	\end{equation}

	(2) Moreover, there are constant (matrics) $a,b,a_i,b_i,\tilde{a}_i,\tilde{b}_i$ such that
	\begin{equation}
		\begin{split}
		A-\tilde{A} & = \left( a+ bt+\sum_{i=1}^4 (\tilde{a}_i e^{\sqrt{3}t}+\tilde{b}_i e^{-\sqrt{3}(t-\log \lambda)})\omega_i \right) dt \\
		& + \sum_{i=1}^4 (a_i e^{\sqrt{3}t}+ b_i e^{-\sqrt{3}(t-\log \lambda)})\psi_i \\
		& + \sum_{i=1}^3 (c_{\pm,i}e^{2t}+ d_{\pm,i}e^{-2(t-\log \lambda)})\phi_{\pm,i}.
		\end{split}
		\label{eqn:expandA}
	\end{equation}
\end{lem}
For the definition of $\psi_i$ and $\phi_{\pm,i}$, we refer to Section \ref{sub:diff}.

\begin{proof}
	Set
\[
	A= f(x,t) dt + \xi(x,t)
\]
where $f$ is a function on $\Omega$ and $\xi(t)$ is one form on $\set{t}\times S^3$ satisfying $\xi(\partial_t)=0$.
Direct computation gives
\begin{equation}
	\label{eqn:directcom}
	\triangle_h A = (- \triangle_{S^3} f - \partial_t^2 f) dt + \left( \triangle_{h;S^3}\xi - \partial_t^2 \xi \right).
\end{equation}
For the reader's convenience, we include the details of the computation in the appendix.

By applying Theorem \ref{thm:linear} with $L$ equals $-\triangle_{S^3}$ and $\triangle_{h;S^3}$ respectively, we obtain the existence of $\tilde{f}$ and $\tilde{\xi}$ satisfying
\[
	\triangle_h (\tilde{f} dt + \tilde{\xi}) = \triangle_h A
\]
and
\begin{equation}
	 \abs{\tilde{f}} + \abs{\tilde{\xi}}\leq C \eta^{\alpha_1}(t), \quad \forall t\in [\log\lambda-\log\delta,\log\delta].
	\label{eqn:tildesmall}
\end{equation}
By elliptic estimate, we obtain \eqref{eqn:another}.

For the proof of (2), we set $H=A-\tilde{A}$, which is a harmonic one form and can be written as 
\begin{equation}
	H:=f_H dt + \xi_H.
	\label{eqn:H}
\end{equation}
By the formula \eqref{eqn:directcom}, $f_H$ is a harmonic function, i.e. $\triangle f_H=0$ and $\xi_H$ satisfies
\[
	\partial_t^2 \xi_H - \triangle_{h;S^3} \xi_H =0.
\]
By the separation of variables, $f_H$ and $\xi_H$ have an expansion involving the eigenfunctions of $\triangle_{S^3}$ and $\triangle_{h;S^3}$. More precisely,
\[
	f_H(t)= a + bt + \sum_{i=1}^4 (\tilde{a}_i e^{\sqrt{3}t}+\tilde{b}_i e^{-\sqrt{3}(t-\log \lambda)})\omega_i + r_f 
\]
and
\[
	\xi_H(t)=\sum_{i=1}^4 (a_i e^{\sqrt{3}t}+ b_i e^{-\sqrt{3}(t-\log \lambda)})\psi_i +  \sum_{i=1}^3 (c_{\pm,i}e^{2t}+ d_{\pm,i}e^{-2(t-\log \lambda)})\phi_{\pm,i}+  r_\xi.
\]
Since $H$ is bounded by $C\epsilon_0$ (see \eqref{eqn:another} and \eqref{eqn:smallA}), the remainder terms $r_f$ and $r_\xi$, that are also harmonic, satisfy
\[
	\norm{r_f}_{C^{\alpha_1}([t])},\norm{r_\xi}_{C^{\alpha_1}([t])}\leq C\epsilon_0 \eta^{\alpha_1}(t).
\]
The above estimate also depends on the fact that $\alpha_1\in (2,\sqrt{8})$. The proof of (2) is concluded by choosing $\tilde{A}+( r_f dt + r_\xi)$ as $\tilde{A}$.
\end{proof}

In the trivialization obtained by Theorem \ref{thm:neckgauge}, the connection form satisfies \eqref{eqn:gg1} and \eqref{eqn:gg2}, in addition to $d^*A=0$. These help us in improving \eqref{eqn:expandA}.
\begin{cor}\label{cor:improveA}
	Let $D$ be a Yang-Mills connection on $\Omega$ satisfying \eqref{eqn:rade}. In the trivialization given by Theorem \ref{thm:neckgauge}, the connection is given by a one-form $A$, for which we have
	\begin{equation}
		A = \sum_{i=1}^3 (c_{\pm,i}e^{2t}+ d_{\pm,i} e^{-2(t-\log \lambda)}) \phi_{\pm,i} + r_A
		\label{eqn:improveA}
	\end{equation}
	where 
	\[
		\norm{r_A}_{C^{\alpha_1}(\Omega_{[t]})}\leq C \eta^{\alpha_1}(t)
	\]
	for all $t\in [\log \lambda -\log\delta +2, \log\delta -2]$.
\end{cor}
\begin{proof}
	Lemma \ref{lem:improve} gives a form $\tilde{A}$ that satisfies \eqref{eqn:expandA}. Notice that if we set $r_A=\tilde{A}$, then it satisfies \eqref{eqn:another} automatically. The point is that why the first two terms in the right hand side of \eqref{eqn:expandA} satisfies \eqref{eqn:another} and hence can be absorbed into $r_A$.

	Combining \eqref{eqn:expandA} and \eqref{eqn:another}, we derive from \eqref{eqn:gg1} that
	\[
		\abs{\tilde{a}_i+\tilde{b}_i} \lambda^{\sqrt{3}/2} \int_{S^3} \abs{\omega_i}^2 \leq C \eta^{\alpha_1}(\frac{1}{2}\log \lambda)\leq C \lambda^{\alpha_1/2}.
	\]
	Similary, if we use \eqref{eqn:gg2}, we obtain
	\[
		\abs{\tilde{a}_i-\tilde{b}_i} \lambda^{\sqrt{3}/2} \int_{S^3} \abs{\omega_i}^2 \leq C \eta^{\alpha_1}(\frac{1}{2}\log \lambda)\leq C \lambda^{\alpha_1/2}.
	\]
	It follows that
	\[
		\abs{\tilde{a}_i e^{\sqrt{3}t} \omega_i},\abs{\tilde{b}_i e^{-\sqrt{3}(t-\log \lambda)} \omega_i} \leq C \eta^{\alpha_1}(t), \qquad \forall t\in [\log \lambda -\log\delta,\log \delta].
	\]
	Hence, the terms $\tilde{a}_i e^{\sqrt{3}t}\omega_i$ and $\tilde{b}_i e^{-\sqrt{3}(t-\log \lambda)}$ can be absorbed into $r_A$. Since $\int_{\Omega_C} dA^*=0$, we have (see \eqref{eqn:dSA})
	\[
		\abs{b}\leq C \lambda^{\alpha_1/2}.
	\]
	As before, the first equation in \eqref{eqn:gg1} gives
	\[
		\abs{a+ b (\frac{1}{2}\log \lambda)}\leq C \lambda^{\alpha_1/2}.
	\]
From the two inequalities above, we find that
\[
	\abs{a+bt}\leq C \eta^{\alpha_1}(t),\quad \forall t\in [\log\lambda -\log\delta,\log\delta].
\]

In summary, if we adjust $r_A$, we can write 
	\begin{equation*}
		\begin{split}
		A & =  \sum_{i=1}^4 (a_i e^{\sqrt{3}t}+ b_i e^{-\sqrt{3}(t-\log \lambda)})\psi_i \\
		& + \sum_{i=1}^3 (c_{\pm,i}e^{2t}+ d_{\pm,i}e^{-2(t-\log \lambda)})\phi_{\pm,i} +r_A.
		\end{split}
	\end{equation*}

	By applying $d^*$ to the above equation, we obtain
	\begin{equation}
		\label{eqn:dAstar}
		3\sum_{i=1}^4 (a_i e^{\sqrt{3}t}+ b_i e^{-\sqrt{3}(t-\log \lambda)})\omega_i + d^* r_A=0.
	\end{equation}
	Here we have used the following idenities, whose proofs are straight forward computations that are collected in the appendix,
	\begin{equation}
		\label{eqn:collect}
		\begin{split}
		d^*(e^{2t} \phi_{\pm,i})=&  d^*(e^{-2(t-\log \lambda)}\phi_{\pm,i})=0\\
		d^*(e^{\sqrt{3}t} \psi_i) =& 3 e^{\sqrt{3}t} \omega_i \\
		d^*(e^{-\sqrt{3}(t-\log\lambda)} \psi_i) =& 3 e^{-\sqrt{3}(t-\log\lambda)} \omega_i.
		\end{split}
	\end{equation}
We may now argue as before by using \eqref{eqn:dAstar} and
	\[
		\int_{\Omega_C} (d^*A) \cdot \omega_i =\int_{\Omega_C} \partial_t(d^*A) \cdot \omega_i =0.
	\]
	to see that the terms involving $a_i$ and $b_i$ can also be absorbed into $r_A$.
\end{proof}

\subsection{The limit connections}
\label{sub:thelimit}

We now discuss the limit of the blow-up sequence. Recall that $D_k$ is the sequence of Yang-Mills fields discussed in the introduction and Theorem \ref{thm:known}. The cylindrical domain is 
\[
	\Omega_k= [\log \lambda_k-\log\delta,\log\delta]\times S^3.
\]
We also denote the shorter cylinder that appears in Lemma \ref{lem:improve} by $\tilde{\Omega}_k$, i.e.
\[
	\tilde{\Omega}_k= [\log \lambda_k-\log\delta+2,\log\delta-2]\times S^3.
\]
The estimates obtained in Lemma \ref{lem:improve} and Corollary \ref{cor:improveA} imply that
\begin{equation}
	A_k= \sum_{i=1}^3 (c_{\pm,k;i} e^{2t} + d_{\pm,i}e^{-2(t-\log \lambda_k)}) \phi_{\pm,k;i} + r_k
	\label{eqn:Ak}
\end{equation}
where $c_{\pm,k;i}$ and $d_{\pm,k;i}$ are uniformly bounded and that there is a constant independent of $k$ such that
\[
	\norm{r_k}_{C^{\alpha_1}(\Omega_{[t]})}\leq C\eta^{\alpha_1}(t)
\]
for all $t\in [\log \lambda_k-\log\delta +2, \log \delta-2]$.

As $k$ goes to infinity, $\tilde{\Omega}_k$ becomes longer and longer and we can not take a limit of the whole domain. Instead, if we focus on one end (left or right) of $\tilde{\Omega}_k$, we obtain a limit on the half cylinder. More precisely, for any $K$ fixed, passing to the limit $k\to \infty$ and using Lemma \ref{lem:improve} as a priori estimate, we obtain a limit defined on $[\log \delta -K,\log \delta-2]$ 
\begin{equation}
A_\infty=  \sum_{i=1}^3 F_{\pm,\infty;i}e^{2t}\phi_{\pm,i} + r_\infty.
	\label{eqn:limFcylinder}
\end{equation}
By taking diagonal subsequence, the limit is defined on $(-\infty,\log \delta]$ and $r_\infty$ satisfies
\[
	\norm{r_\infty}_{C^{\alpha_1}( \Omega_{[t]})} \leq C\epsilon_0 e^{\alpha_1 t}
\]
for all $t\in (-\infty,\log \delta-2]$.

Similarly, we consider the restriction to $[\log \lambda_k -\log \delta+2, \log \lambda_k-\log \delta +K]$. After a translation by $K$, the connection form $A_k(t-\log \lambda_k)$ converges to 
\begin{equation}
	A_b= \sum_{i=1}^3 F_{\mp,b;i}e^{-2t}\phi_{\pm,i} + r_b \quad \text{on} \quad [-\log\delta+2,\infty )\times S^3,
	\label{eqn:limFbubble}
\end{equation}
where $r_b$ satisfies
\[
	\norm{r_b}_{C^{\alpha_1}(\Omega_{[t]})}\leq C\epsilon_0 e^{-\alpha_1 t}.
\]
\begin{rem}
	We have introduced the notation $F_{\pm,\infty;i}$ and $F_{\mp,b;i}$ instead of $c_{\pm,\infty;i}$ and $d_{\pm,b;i}$. This will be justified in a minute.
\end{rem}

Finally, if we use the normal coordinates $x_1,\dots,x_4$ instead of its cylinder coordiantes and compute the curvature at $x=0$, the equations \eqref{eqn:limFcylinder} yields
\begin{eqnarray*}
	F_\infty(0) &=& \sum_{i=1}^3 2 F_{\pm,\infty;i} \Phi_{\pm,i}.
\end{eqnarray*}
Here we have used 
\[
	\frac{1}{2} d(e^{2t}\phi_{\pm,i}) = \pi^*\Phi_{\pm,i}.
\]
Recall that $\pi(t,\omega)=e^t\omega$ is the coordinate change map from cylinder coordinates to Euclidean coordinates and $\Phi_{\pm,i}$ are standard basis of SD/ASD forms on $\Real^4$(see \eqref{eqn:Phiplus} and \eqref{eqn:Phiminus} for the definition). 

For the bubble curvature, we use an coordinate at the infinity given by
\[
(y_1,y_2,y_3,y_4)= \frac{1}{\abs{x}^2} (x_1,x_2,x_3,x_4).
\]
The advantage of this coordinate system is that its associated cylinder coordinates are easily related to the cylinder coordinates $(t,\omega)$. It suffices to add a minus sign in front of $t$. The disadvantage is that the natural orientation of $(y_1,\dots,y_4)$ is different from that of $(x_1,\dots,x_4)$. We use $\Phi_{\pm;i}(y)$ to denote the two forms $\Phi_{\pm;i}$ in $y$-coordinates. Notice that $\Phi_{+;i}(y)$ are ASD forms due to the problem of orientation mentioned above.

With these in mind, we compute the curvature of the bubble connection at $y=0$ in $y$-coordinates to see
\[
F_b(0)= \sum_{i=1}^3 2 F_{\mp,b;i} \Phi_{\pm,i}(y).
\]

To conclude this section, we notice the key fact that will be used in the proof of Theorem \ref{thm:main} is that
\begin{equation}
	F_{\pm,\infty;i}= \lim_{k\to \infty} c_{\pm,k;i}; \quad F_{\mp,b;i}=\lim_{k\to \infty} d_{\pm,k;i}.
	\label{eqn:keyfact}
\end{equation}

\section{Proof of the main theorem}
\label{sec:proof}

Since we have assumed that the metric $g$ is conformally flat near the only blow-up point $\bx_0$. We may assume it is flat there and we take a coordinate system $x_1,\dots,x_4$ in which the metric $g=\delta_{ij}dx_i\otimes dx_j$. Due to the conformal invariance, we may also regard $D_k$ as a Yang-Mills connection on the cylinder $\Omega_k$ with repsct to the round cylinder metric. Hence if $S$ is the stress-energy tensor (as in \eqref{eqn:stress} w.r.t the round cylinder metric), then
\[
	\mbox{div} S=0.
\]

If $X$ is a vector field on cylinder that generates rotation, or $X=\partial_t$, we have
\[
	\mbox{div} (S\# X)=0.
\]
The diverngence theorem over $(-\infty,\frac{1}{2}\log \lambda]\times S^3$ implies that
\[
	\int_{\Omega_C} S \# X \# \partial_t d\omega=0,
\]
since $\lim_{t\to -\infty }\abs{F}=0$.
More precisely,
\begin{equation}
	\int_{\Omega_C} \langle \iota_X F, \iota_{\partial_t} F \rangle - \frac{1}{4}\abs{F}^2 \langle X, \partial_t \rangle d\omega=0.
	\label{eqn:poho}
\end{equation}

To prove Theorem \ref{thm:main}, we take the gauge in Theorem \ref{thm:neckgauge} and substitute our expansion of $A$ in Corollary \ref{cor:improveA} into \eqref{eqn:poho}. Indeed, dropping the subscript $k$ for simplicity, we may write $A_k$ as 
\begin{equation}
	A= \sum_{i=1}^3 (F^L_{\pm;i} e^{2t} + F^R_{\mp;i} e^{-2(t-\log \lambda)})\phi_{\pm,i} + r_A
	\label{eqn:rewrite}
\end{equation}
where $\abs{r_A}\leq C\eta^{\alpha_1}(t)$.
Here we have used $F^L_{\pm;i}$ for $c_{\pm,i}$ and $F^R_{\mp;i}$ for $d_{\pm;i}$. We will compute \eqref{eqn:poho} for $X= X_{\pm,i}$ and for $X=\partial_t$ separately. The proof will eventually be finished by taking a limit $\lambda\to 0$. It is the leading term, i.e. the term of order $\lambda^2$, that matters. Hence in the following computations, any term that contributes to a term of higher order will be ignored. In this sense, we write $\sim$ instead of $=$.

{\bf Case 1:} $X=X_{\pm;i_0}$. In this case $\langle X, \partial_t \rangle =0$ and it suffices to compute
\[
	\int_{\Omega_C} \langle \iota_X F, \iota_{\partial_t} F \rangle d\omega.
\]
When restricted to $\Omega_C$, $\abs{A}, \abs{dA}$ are bounded by $C\lambda$ and $X$ are bounded. Hence all terms involving $[A\wedge A]$ and $r_A$ do not contribute to the $\lambda^2$ term. Omitting these higher order terms, we write
\begin{eqnarray*}
	\int_{\Omega_C} \langle \iota_X F, \iota_{\partial_t} F \rangle d\omega &\sim& \int_{\Omega_C} \langle \iota_X (dA), \iota_{(\partial_t)}(dA) \rangle d\omega.
\end{eqnarray*}
Taking $d$ of \eqref{eqn:rewrite} and omitting terms smaller than $\lambda$, we obtain
\[
	dA\sim \sum_{i=1}^3 \left( F^L_{\pm;i} d(e^{2t}\phi_{\pm;i}) + F^R_{\mp;i} \lambda^2 d(e^{-2t}\phi_{\pm;i}) \right).
\]
Recalling the definition of $\mathcal P_{\pm;i}$ and $\mathcal Q_{\pm;i}$ (see \eqref{eqn:defpq}) and evaluating at $\Omega_C$ (where $e^{2t}=\lambda$), we have
\begin{equation}
dA|_{\Omega_C}\sim \lambda \sum_{i=1}^3 \left( F^L_{\pm;i} \mathcal P_{\pm;i} + F^R_{\mp;i}  \mathcal Q_{\pm;i}\right).
	\label{eqn:dA}
\end{equation}

We now compute $\iota_X dA|_{\Omega_C}$ and summarize the results in the form of a lemma.
\begin{lem}
	\label{lem:compute}
	\begin{equation}
	 \iota_{X_{+,j}} dA|_{\Omega_C} \stackrel{ {\rm mod} dt }{\sim} \lambda \left( 
	(F^L_{+;i} + F^R_{-;i})+ (F^R_{+;l}+ F^L_{-;l})(T)_{il} 
	\right) \iota_{X_{+;j}} \mathcal P_{+,i}.
	\label{eqn:xjgood}
	\end{equation}
	\begin{equation}
	 \iota_{X_{-,j}} dA|_{\Omega_C} \stackrel{ {\rm mod} dt }{\sim} \lambda \left( 
	(F^L_{-;i} + F^R_{+;i})+ (F^R_{-;l}+ F^L_{+;l})(T)_{il} 
	\right) \iota_{X_{-;j}} \mathcal P_{-,i}.
	\label{eqn:xjminus}
	\end{equation}
	and
	\begin{equation}
	\iota_{\partial_t} (dA)|_{\Omega_C}\sim 2\lambda \sum_{i=1}^3 \left( (F^L_{+;i}-F^R_{-;i}) + (F^R_{+;l}-  F^L_{-;l}) (T)_{il}) \right) \phi_{+;i}.
		\label{eqn:dtgood}
	\end{equation}
\end{lem}
\begin{rem}
	Notice that $\iota_{\partial_t}(dA)|_{\Omega_C}$ is (up to higher order term) a vector field tangent to $S^3$. Since in \eqref{eqn:poho} we care only about the inner product of $\iota_{X_{\pm;j}} dA$ and $\iota_{\partial_t} dA$, the $dt$ terms are omitted in the right hand side of \eqref{eqn:xjgood} and \eqref{eqn:xjminus}.
\end{rem}
\begin{proof}
	By \eqref{eqn:dA} and Lemma \ref{lem:pq}, we have
\begin{eqnarray*}
	&&\iota_{X_{+;j}} dA|_{\Omega_C}\\
	&\sim& \lambda \sum_{i=1}^3 \left( F^L_{+;i} \iota_{X_{+;j}} \mathcal P_{+;i} +  F^L_{-;i} \iota_{X_{+;j}} \mathcal P_{-;i} +   F^R_{-;i} \iota_{X_{+;j}} \mathcal Q_{+;i}  +  F^R_{+;i} \iota_{X_{+;j}} \mathcal Q_{-;i}\right)\\
	&\sim& \lambda \sum_{i=1}^3\left[ \left( F^L_{-;i}+ F^R_{-;l}(T)_{li} \right) \iota_{X_{+;j}} \mathcal P_{-;i} +  \left( F^L_{+;i} + F^R_{+;l}(T)_{il} \right) \iota_{X_{+;j}} \mathcal P_{+;i}\right].
\end{eqnarray*}
By Lemma \ref{lem:PQ}, 
\begin{equation*}
	\iota_{X_{+;j}} \mathcal P_{-;i} \stackrel{{\rm mod} dt}{=} \iota_{X_{+;j}} \mathcal Q_{-;i} \stackrel{{\rm mod} dt}{=} (T)_{li} \iota_{X_{+;j}} \mathcal P_{+,l}.
\end{equation*}
Hence,
\begin{equation}
	\begin{split}
	& \iota_{X_{+,j}} dA|_{\Omega_C} \\
	\stackrel{ {\rm mod} dt }{\sim} & \lambda \left( 
	F^L_{+;i} + F^R_{+;l}(T)_{il} + F^L_{-;m}(T)_{im} + F^{R}_{-;l}(T)_{lm}(T)_{im}
	\right) \iota_{X_{+;j}} \mathcal P_{+,i}\\
	\stackrel{ {\rm mod} dt }{\sim} & \lambda \left( 
	(F^L_{+;i} + F^R_{-;i})+ (F^R_{+;l}+ F^L_{-;l})(T)_{il} 
	\right) \iota_{X_{+;j}} \mathcal P_{+,i}.
	\end{split}
	\label{eqn:xj}
\end{equation}
This is \eqref{eqn:xjgood} and the proof of \eqref{eqn:xjminus} is the same.
For \eqref{eqn:dtgood}, we compute
\begin{eqnarray*}
	\iota_{\partial_t} dA|_{\Omega_C} &\sim& \lambda \sum_{i=1}^3 \left( F^L_{\pm;i} \iota_{\partial_t} \mathcal P_{\pm;i} + F^R_{\mp;i} \iota_{\partial_t} Q_{\pm;i} \right)\\ 
					  &\sim & \lambda \sum_{i=1}^3 \left[ \left( F^L_{+;i} + F^R_{+;l}(T)_{il} \right) \iota_{\partial_t} \mathcal P_{+;i} +  \left( F^L_{-;i} + F^R_{-;l}(T)_{li} \right)\iota_{\partial_t} \mathcal P_{-;i} \right]\\
					  &=& 2\lambda \sum_{i=1}^3 \left[ \left( F^L_{+;i} + F^R_{+;l}(T)_{il} \right) \phi_{+;i} + \left( F^L_{-;i} + F^R_{-;l}(T)_{li} \right) \phi_{-;i} \right].
\end{eqnarray*}
Here in the last line above, we have used the first row of \eqref{eqn:table}.

Since $\phi_{-;i}=-(T)_{ji}\phi_{+;j}$ (see \eqref{eqn:phipm}), we simplify the above
\begin{equation}
	\begin{split}
	\iota_{\partial_t} (dA)|_{\Omega_C} &\sim 2\lambda \sum_{i=1}^3 \left( F^L_{+;i} + F^R_{+;l}(T)_{il}+  (F^L_{-;j} + F^R_{-;l}(T)_{lj})(-(T)_{ij}) \right) \phi_{+;i} \\
	&\sim 2\lambda \sum_{i=1}^3 \left( (F^L_{+;i}-F^R_{-;i}) + (F^R_{+;l}-  F^L_{-;l}) (T)_{il}) \right) \phi_{+;i} 
	\end{split}
	\label{eqn:dt}
\end{equation}
\end{proof}

With Lemma \ref{lem:compute}, we continue to compute the left hand side of \eqref{eqn:poho} for the case $X=X_{+,1}$. The other cases are proved in the same way.
Taking $j=1$ in \eqref{eqn:xjgood} and using \eqref{eqn:table}, we get 
\begin{eqnarray*}
	\iota_{X_{+;1}} dA|_{\Omega_C} &\stackrel{ {\rm mod}dt}{\sim}& 2\lambda\left( 
	(F^L_{+;2} + F^R_{-;2})+ (F^R_{+;l}+ F^L_{-;l})(T)_{2l} 
	\right) (-\phi_{+;3}) \\
	&&+ 2\lambda \left( 
	(F^L_{+;3} + F^R_{-;3})+ (F^R_{+;l}+ F^L_{-;l})(T)_{3l}
	\right) \phi_{+;2}.
\end{eqnarray*}
Since $\set{\phi_{+,i}}_{i=1,2,3}$ is an orthnormal basis, we have
\begin{eqnarray*}
	&& \langle \iota_{\partial_t}dA ,\iota_{X_{+;1}} dA \rangle \\
	&\sim& 4\lambda^2 \langle (F^L_{+;2}-F^R_{-;2}) + (F^R_{+;l}-  F^L_{-;l}) (T)_{2l}, (F^L_{+;3} + F^R_{-;3})+ (F^R_{+;l}+ F^L_{-;l})(T)_{3l} \rangle \\
	&-& 4\lambda^2 \langle (F^L_{+;3}-F^R_{-;3}) + (F^R_{+;l}-  F^L_{-;l}) (T)_{3l}, (F^L_{+;2} + F^R_{-;2})+ (F^R_{+;l}+ F^L_{-;l})(T)_{2l} \rangle.
\end{eqnarray*}
Using \eqref{eqn:xy}, we can check the following fact
\begin{equation}
	\label{eqn:Tfact}
	\int_{S^3} T_{ij} d\omega=0; \quad \int_{S^3} T_{ij} T_{i'j'}d\omega=0
\end{equation}
as long as $(i,j)\ne (i',j')$. 
Therefore,
\begin{eqnarray*}
	&& \int_{\Omega_C} \langle \iota_{X_{+;1}} dA, \iota_{\partial_t} dA \rangle \\
	&\sim& 4\lambda^2\int_{\Omega_C} \langle (F^L_{+;2}-F^R_{-;2}) , (F^L_{+;3} + F^R_{-;3}) \rangle - \langle (F^L_{+;3}-F^R_{-;3})  , (F^L_{+;2} + F^R_{-;2}) \rangle \\
	&\sim& 8\lambda^2 \abs{S^3} \left(\langle F^L_{+;2}, F^R_{-;3} \rangle - \langle F^L_{+;3}, F^R_{-;2} \rangle \right).
\end{eqnarray*}
Now, we reinstall the subscript $k$ and take the limit of $k\to \infty $ in the equation \eqref{eqn:poho}(divided by $\lambda^2$), to see
\[
\lim_{k\to \infty} \left(\langle F^L_{+,k;2}, F^R_{-,k;3} \rangle - \langle F^L_{+,k;3}, F^R_{-,k;2} \rangle \right)=0.
\]
Similar computation for $j=2$ and $j=3$ gives
\[
\lim_{k\to \infty} \left(\langle F^L_{+,k;3}, F^R_{-,k;1} \rangle - \langle F^L_{+,k;1}, F^R_{-,k;3} \rangle \right)=0
\]
and
\[
\lim_{k\to \infty} \left(\langle F^L_{+,k;1}, F^R_{-,k;2} \rangle - \langle F^L_{+,k;2}, F^R_{-,k;1} \rangle \right)=0.
\]
Together with \eqref{eqn:keyfact}, we finish the proof of \eqref{eqn:main16} in Theorem \ref{thm:main}.

{\bf Case 2.} $X=\partial_t$. 

In this case, \eqref{eqn:poho} reads
\begin{equation}
	\label{eqn:poho2}
\int_{\Omega_C} \abs{\iota_{\partial_t} F_k}^2 -\frac{1}{4} \abs{F_k}^2 d\omega =0.
\end{equation}
As before, we concentrate on terms of scale $\lambda^2$. 

Recall that $\set{\phi_{+;i}}_{i=1}^3$ is an orthonormal basis. Therefore (dropping the subscript $k$ again)
\begin{equation}
	\label{eqn:case21}
	\begin{split}
	\int_{\Omega_C} \abs{\iota_{\partial_t} F}^2 \sim& \int_{\Omega_C} \abs{\iota_{\partial_t} dA}^2 \\
	\sim& 4\lambda^2 \sum_{i=1}^3\int_{\Omega_C} \abs{(F^L_{+;i}-F^R_{-;i}) + (F^R_{+;l}-  F^L_{-;l}) (T)_{il}}^2\\
	=& 4\lambda^2 \sum_{i=1}^3\int_{\Omega_C} \abs{F^L_{+;i}-F^R_{-;i}}^2 + \abs{F^R_{+;i}-  F^L_{-;i}}^2 \\
	=& 4\lambda^2 \int_{\Omega_C} \sum_{i=1}^3 \abs{F^L_{+;i}}^2+\abs{F^L_{-;i}}^2+\abs{F^R_{+;i}}^2+\abs{F^R_{-;i}}^2 \\
	&- 8\lambda^2 \sum_{i=1}^3 \int_{\Omega_C} \langle F^L_{+;i}, F^R_{-;i} \rangle + \langle F^L_{-;i}, F^{R}_{+;i} \rangle.
	\end{split}
\end{equation}
Here we have used \eqref{eqn:Tfact} again and the fact that $T$ is an orthogonal matrix.
On the other hand,
\begin{eqnarray*}
	\int_{\Omega_C} \abs{F}^2 &\sim& \int_{\Omega_C} \abs{dA}^2.
\end{eqnarray*}
By Lemma \ref{lem:pq} and \eqref{eqn:dA},
\begin{equation*}
	dA|_{\Omega_c} \sim \lambda \sum_{i=1}^3 \left[ (F^L_{+;i}+ F^R_{+;j}(T)_{ij}) \mathcal P_{+;i} +  (F^L_{-;i}+ F^R_{-;j}(T)_{ji}) \mathcal P_{-;i} \right].
\end{equation*}
By the part (3) of Lemma \ref{lem:pq}, 
\begin{equation}
	\label{eqn:case22}
	\begin{split}
	\int_{\Omega_C} \abs{dA}^2 \sim& (4\lambda)^2 \int_{\Omega_C}\sum_{i=1}^3 \left( \abs{ F^L_{+;i}+ F^R_{+;j}(T)_{ij}}^2  + \abs{F^L_{-;i}+ F^R_{-;j}(T)_{ji}}^2\right) \\
	=& 16\lambda^2 \int_{\Omega_C} \sum_{i=1}^3 \abs{F^L_{+;i}}^2 +\abs{F^L_{-;i}}^2+\abs{F^R_{+;i}}^2 +\abs{F^R_{-;i}}^2.
	\end{split}
\end{equation}
Plugging \eqref{eqn:case21} and \eqref{eqn:case22} into \eqref{eqn:poho2} and comparing the coefficients of $\lambda^2$, we obtain
\[
\lim_{k\to \infty} \sum_{i=1}^3 \langle F^L_{+,k;i}, F^R_{-,k;i} \rangle + \langle F^L_{-,k;i}, F^{R}_{+,k;i} \rangle =0.
\]
This completes the proof of \eqref{eqn:main7}.

\appendix
\section{A linear equation on cylinder}
Let $E$ be some vector bundle over $S^3$ equipped with an inner product.
Let $L$ be an self-adjont, elliptic, nonnegative operator defined on sections of $E$. Let $u(t)$ and $f(t)$ be two families of sections of $E$ satisfying
\begin{equation}
	(\partial_t^2- L) u(t)=f(t)
	\label{eqn:Lequ}
\end{equation}
Assume the eigenvalues of $L$ are $\lambda_1^2,\dots,\lambda_n^2,\dots$ in increasing order and the corresponding eigenfunctions are $\varphi_i$, normalized with respect to $L^2$ norm. Here $\lambda_i$'s are counted with multiplicity and are assumed to be nonnegative.

The main result of this section is
\begin{thm}\label{thm:linear}
	Suppose $\alpha$ is a positive number different from any $\lambda_i$. If $f$ is some smooth function satisfying $\abs{f}\leq e^{-\alpha M} e^{\alpha \abs{t}}$ for $t\in [-M,M]$, there is a smooth solution $u(t)$ to \eqref{eqn:Lequ} such that
	\[
		\abs{u}\leq C e^{-\alpha M} e^{\alpha \abs{t}},\quad \text{for} \quad t\in [-M,M].
	\]
	Here $C$ depends on $\alpha$, but not on $M$.
\end{thm}

\begin{rem}
	For our purposes, we are interested in two simple cases. Either $u(t)$ is a function on $S^3$, or it is a one form on $S^3$. The operator $L$ is Laplace and Hodge Laplace respectively.
\end{rem}

We expand $f$ as a series using eigenfunctions(or forms)
\[
	f(t)= \sum_{i=1}^\infty a_i(t) \varphi_i.
\]
It follows from Parsaval equality that
\begin{equation}
	\sum_{i=1}^\infty a_i^2(t)\leq C e^{-2\alpha M} e^{2\alpha \abs{t}}.
	\label{eqn:parsaval}
\end{equation}

Formally, if the solution $u$ has an expansion
\[
	u(t)= \sum_{i=1}^\infty A_i(t) \varphi_i,
\]
then $A_i$ is a solution to the ODE
\begin{equation}
	\partial_t^2 A_i- \lambda_i^2 A_i= a_i.
	\label{eqn:ode}
\end{equation}

The proof of the theorem consists of two parts. First, find $A_i$ for each $i$ and then show that $\sum_i A_i$ converges and yields the solution we want.

{\bf Step 1.} Find solution $A_i$ for each $\lambda_i$. The solutions to \eqref{eqn:ode} is not unique. It turns out that we need to choose a particular $A_i$. We make different choices for different $\lambda_i$.

{\bf Case 1.} $\lambda_1=0$. (If $\lambda_1>0$, then skip this part.)
Set
\[
	A_1(t)= \int_0^t \int_0^s a_1(x)dxds.
\]
By \eqref{eqn:parsaval},
\begin{eqnarray*}
	A_1^2(t)&\leq& C \left( \int_{0}^{\abs{t}} \int_0^{\abs{s}} e^{-\alpha M} e^{ \alpha \abs{x}} dx ds \right)^2\\
	&\leq& C e^{-2\alpha M} e^{2\alpha \abs{t}}.
\end{eqnarray*}

{\bf Case 2.} $\lambda_i< \alpha$. There are only finitely many $\lambda_i$'s that are smaller than $\alpha$.
For $t\in [-M,M]$, set
\begin{equation}
	A_i(t)= e^{\lambda_i t} \int_0^t e^{-2\lambda_i s} \left( \int_0^s a_i(x)e^{\lambda_i x} dx \right) ds
	\label{eqn:ai1}
\end{equation}
If $t>0$, we estimate
\begin{eqnarray*}
	\abs{A_i(t)}&\leq& e^{\lambda_i t}\int_0^t e^{-2\lambda_i s} \left( \int_0^s C e^{-\alpha M} e^{\alpha x}e^{\lambda_i x} dx \right) ds \\
	&\leq& C e^{-\alpha M} e^{\alpha t}.
\end{eqnarray*}
If $t<0$, we write $\tilde{t}=-t$ and estiamte as before (using $\alpha>\lambda_i$)
\begin{eqnarray*}
	\abs{A_i(t)}&=&\abs{ e^{- \lambda_i \tilde{t}}\int_0^{\tilde{t}} e^{2\lambda_i s} \left( \int_0^s a_i(-x) e^{-\lambda_i x} dx \right) ds} \\
	&\leq& C e^{-\alpha M} e^{\alpha \abs{t}}.
\end{eqnarray*}

{\bf Case 3.} $\lambda_i>\alpha$.

Notice that the following is also a solution of \eqref{eqn:ode}
\begin{equation}
	A_i(t)=	-\frac{1}{2\lambda_i} e^{\lambda_i t} \int_t^{\infty} e^{-\lambda_i x} a_i (x) dx - \frac{1}{2\lambda_i} e^{-\lambda_i t} \int_{-\infty}^t e^{\lambda_i x} a_i(x) dx.
	\label{eqn:ai2}
\end{equation}
Here we have extended the domain of $a_i$ by setting $a_i(x)=0$ for $\abs{x}>M$.
Hence,
\begin{eqnarray*}
	\abs{A_i}^2(t) &\leq & \frac{1}{2\alpha^2} e^{2\lambda_i t} \left( \int_t^\infty e^{-\lambda_i x} a_i(x)dx \right)^2 + \frac{1}{2\alpha^2} e^{-2\lambda_i t} \left( \int_{-\infty}^t e^{\lambda_i x} a_i(x) \right)^2 \\
&:=& I+ II.	
\end{eqnarray*}
We assume $t>0$. (The case $t<0$ is estimated similarly.) Choose $\lambda'$ and $\delta$ satisfying $\alpha+\delta<\lambda'<\lambda_i$ for all $\lambda_i>\alpha$ so that
\begin{eqnarray*}
	I &\leq& C(\alpha) e^{2\lambda_i t} \left( \int_t^\infty e^{-2\delta x} dx \right) \left( \int_t^\infty e^{2\delta x} e^{-2\lambda_i x} a_i^2(x) dx \right) \\
	&\leq& C(\alpha) e^{-2\delta t} \left( \int_t^\infty e^{2\delta x} e^{2\lambda_i (t-x)} a_i^2(x)dx \right) \\
	&\leq&  C(\alpha) e^{-2\delta t} \int_t^\infty e^{2\delta x} e^{2 \lambda'(t-x)} a_i^2(x) dx.
\end{eqnarray*}
Here in the last line above, we used the fact that $t-x<0$.
For $II$, 
\begin{eqnarray*}
	II &\leq& C(\alpha) e^{-2\lambda_i t} \left( \int_{-\infty}^0 e^{\lambda_i x} a_i(x) dx \right)^2 + C(\alpha) e^{-2\lambda_i t} \left( \int_0^t e^{\lambda_i x} a_i(x) dx\right)^2 \\
	   &\leq& C(\alpha,\delta) e^{-2\lambda_i t} \int_{-\infty}^0 e^{\red{-}2\delta x}e^{2\lambda_i x} a_i^2(x) dx + C(\alpha,\delta) e^{-2\lambda_i t} e^{2\delta t} \int_0^t e^{-2\delta x} e^{2\lambda_i x} a_i^2(t) dx \\
	&\leq& C \int_{-\infty}^0 e^{-2\delta x}e^{-2\lambda' (t-x)} a_i^2(x) dx + C  e^{2\delta t} \int_0^t e^{-2\delta x} e^{-2\lambda' (t-x)} a_i^2(t) dx.
\end{eqnarray*}
Here in the last line above, we used the fact that $t-x>0$. 

{\bf Step 2.}
We combine all estimates obtained above to bound $\sum_i A_i^2(t)$ and keep in mind that there are finitely many (depending on $\alpha$) $\lambda_i$ that is smaller than $\alpha$. Hence
\begin{eqnarray*}
	\sum_i A_i^2(t) &=&  \sum_{\lambda_i<\alpha} A_i^2(t) + \sum_{\lambda_i>\alpha} A_i^2(t) \\
	&\leq& C e^{-2\alpha M} e^{2\alpha \abs{t}} + \sum_i (I + II).
\end{eqnarray*}
Recall that $\lambda'>\alpha+\delta$, which implies that 
\begin{eqnarray*}
	\sum_i I &\leq& C e^{-2\lambda M} e^{-2\delta t} \int_t^\infty e^{2\delta x} e^{2\lambda'(t-x)} e^{2\alpha x} dx \\
	&\leq& C e^{-2\alpha M} e^{2\alpha t}.
\end{eqnarray*}
Similarly,
\begin{eqnarray*}
	\sum_i II &\leq& C e^{-2\alpha M} \int_{-\infty}^0 e^{-2\delta x} e^{-2\lambda'(t-x)} e^{2\alpha \abs{x}} dx  \\
	&& + C e^{-2\lambda M} e^{2\delta t} \int_0^t e^{-2\delta x} e^{-2\lambda'(t-x)} e^{2\alpha x} dx \\
	&\leq& Ce^{-2\lambda M} \left( e^{-2\lambda' t} + e^{2\alpha t} \right)\\
	&\leq& C e^{-2\alpha M} e^{2\alpha t}.
\end{eqnarray*}
Finally, the above estimate implies that 
\[
	u_N:=\sum_{i=1}^N A_i(t) \varphi_i
\]
converges to $u=\sum_{i=1}^\infty A_i(t)\varphi_i$ locally in $L^2$. Moreover, $u_N$ is a classical solution to the equation
\[
	(\partial_t^2-L)u_N=f_N
\]
where
\[
	f_N:=\sum_{i=1}^N a_i(t)\varphi_i.
\]
Since $f_N$ converges to $f$ smoothly as $N\to \infty$, the elliptic estimates imply that $u_N$ converges to $u$ smoothly. Hence $u$ is a classical solution to \eqref{eqn:Lequ}.

\section{Elliptic estiamte on cylinder}
In this section, we prove two elliptic estimates on long cylinder $\Omega$ corresponding to Neumann and Dirichlet boundary conditions respectively.

\begin{lem}\label{lem:neumann}
	Let $u$ be a $C^{\alpha_1}$ function on $\Omega$ satisfying
	\[
		\triangle u =v; \quad \Psi(\partial_t u)|_{\Omega_L}=v_L; \quad \Psi(\partial_t u)|_{\Omega_R}=v_R
	\]
	and the normalization condition $\int_{\Omega_C} u =0$. Set 
	\begin{equation}
		\label{eqn:bab}
		\int_{\Omega_c} \partial_t u = b_0; \quad \int_{\Omega_C} \partial_t u \cdot \omega_i = \bar{a}_i; \quad \int_{\Omega_C} \partial^2_t u \cdot \omega_i = \bar{b}_i.
	\end{equation}
	Then
	\begin{equation}
		\label{eqn:neumann}
		\norm{u}_{\mathcal X_1} \leq C \left( \norm{v}_{\mathcal X_3} + \norm{v_L}_{C^{\alpha_2}(\Omega_L)}+ \norm{v_R}_{C^{\alpha_2}(\Omega_R)}+ \frac{\abs{b_0} + \sum_{i=1}^4 (\abs{\bar{a}_i}+\abs{\bar{b}_i})}{\lambda^{\alpha_2/2}}\right).
	\end{equation}
\end{lem}

\begin{proof}
	{\bf Step 1.} We may assume that $v$ is identically zero on $[\log\lambda -\log \delta +5, \log \delta -5]$. In fact, by Theorem \ref{thm:linear}, we have $\tilde{u}$ defined on $\Omega$ satisfying
	\[
		\triangle \tilde{u}=v
	\]
	and
	\[
		\abs{\tilde{u}} \leq C \norm{v}_{\mathcal X_3} \eta^{\alpha_2}(t), \quad \forall t\in [\log \lambda -\log \delta +1, \log \delta -1].
	\]
	By the interior estimate of the Poisson equation, we have
	\begin{equation}
		\label{eqn:intu}
		\norm{\tilde{u}}_{C^{\alpha_1}(\Omega_{[t]})}\leq C \norm{v}_{\mathcal X_3} \eta^{\alpha_2}(t), \quad \forall t\in [\log \lambda -\log \delta +3, \log \delta -3].
	\end{equation}
	Let $\phi$ be the smooth cut-off function satisfying
	\[
		\phi(t)= \left\{
			\begin{array}[]{ll}
				1 & t\in [\log \lambda - \log\delta +4, \log \delta -4] \\
				0 & t\in [\log \lambda -\log \delta, \log \lambda -\log \delta + 7/2]\cup [\log\delta-7/2,\log \delta].
			\end{array}
			\right.
	\]
	It follows from \eqref{eqn:intu} that
	\begin{equation}
		\norm{\phi \tilde{u}}_{\mathcal X_1}\leq C \norm{v}_{\mathcal X_3}. 
		\label{eqn:cutmiddle}
	\end{equation}

	If we set
	\[
		\underline{u}:=u-\phi \tilde{u}-a
	\]
	where $a$ is some constant that makes $\int_{\Omega_C} \underline{u}=0$,
	then $\triangle \underline{u}$ is identically zero for $t\in [\log \lambda-\log\delta+4,\log\delta-4]\times S^1$. It suffices to prove the lemma for $ \underline{u}$, because
	\begin{itemize}
		\item The size of $a$ is estimated by
	\[
		\abs{a}\leq C\int_{\Omega_C} \abs{\tilde{u}} \leq C \norm{v}_{\mathcal X_3} \lambda^{\alpha_2/2},
	\]
	which implies that $\norm{a}_{\mathcal X_1}\leq C \norm{v}_{\mathcal X_3}$ and hence $\norm{u-\underline{u}}_{\mathcal X_1}\leq C\norm{v}_{\mathcal X_3}$.
\item  
	\[
		\Psi(\partial_t \underline{u})= \Psi(\partial_t u)\quad  \text{on} \quad  \Omega_L\cup \Omega_R.
	\]
\item 
\[
	\norm{\triangle \underline{u}}_{\mathcal X_3} \leq C \norm{v}_{\mathcal X_3}.
\]
\item By \eqref{eqn:cutmiddle}, we have
	\[
		\abs{\int_{\Omega_C}\partial_t (u- \underline{u})}, \abs{\int_{\Omega_C}\partial_t (u- \underline{u})\cdot \omega_i}, \abs{\int_{\Omega_C}\partial^2_t (u- \underline{u})\cdot \omega_i} \leq C \norm{v}_{\mathcal X_3} \lambda^{\alpha_2/2}.
	\]
	\end{itemize}

	{\bf Step 2.} For the proof of this lemma, we may assume that $b_0= \bar{a}_i= \bar{b}_i=0$. Otherwise, we may find $a_i,b_i$ and subtract from $u$ a harmonic function $h$ given by
	\[
		h=b_0(t-\frac{1}{2}\log \lambda) + \sum_{i=1}^4 (a_i e^{\sqrt{3}t} + b_i e^{-\sqrt{3}(t-\log \lambda)})\omega_i.
	\]
	For an estimate of $a_i$ and $b_i$, we derive from
	\[
		\int_{\Omega_C} \partial_t h \cdot \omega_i =\bar{a}_i; \quad \int_{\Omega_C}  \partial^2_t h \cdot \omega_i=\bar{b}_i
	\]
	that
	\[
		\sqrt{3} (a_i-b_i)\lambda^{\sqrt{3}/2} \int_{S^3} \omega_i^2 = \bar{a}_i; \quad 3(a_i+b_i) \lambda^{\sqrt{3}/2} \abs{S^3} \omega_i^2 = \bar{b}_i.
	\]
	Hence,
	\[
		\sum_{i=1}^4\left( \abs{a_i}+\abs{b_i} \right)\leq C \lambda^{-\sqrt{3}/2} \sum_{i=1}^4 \left( \abs{\bar{a}_i} + \abs{\bar{b}_i} \right).
	\]
With this, we can verify by definition that 
\[
	\norm{h}_{\mathcal X_1}\leq C \frac{\abs{b_0} + \sum_{i=1}^4 (\abs{\bar{a}_i}+\abs{\bar{b}_i})}{\lambda^{\alpha_2/2}}.
\]
Notice that subtracting $h$ does not affect the equation and the boundary condition at $\Omega_L\cup \Omega_R$.

{\bf Step 3.} 
Following the previous step, we assume that $b_0=\bar{a}_i=\bar{b}_i=0$. The aim of this step is prove
\begin{equation}
	\norm{\partial_t u}_{C^{\alpha_2}(\Omega_L\cup \Omega_R)}\leq C \left( \norm{v_L}_{C^{\alpha_2}(\Omega_L)} + \norm{v_R}_{C^{\alpha_2}(\Omega_R)} + \norm{v}_{\mathcal X_3} \right).
	\label{eqn:boundaryut}
\end{equation}
There is a universal constant $C$ such that for any $C^{\alpha_2}$ function $f$ on $S^3$, we have
\begin{equation}
	\label{eqn:unif}
	\norm{f}_{C^{\alpha_2}(S^3)}\leq C\left(\norm{\Psi(f)}_{S^3}+ \abs{\int_{S^3} f} + \sum_{i=1}^4 \abs{\int_{S^3}f\cdot \omega_i}\right).
\end{equation}
Hence, it suffices to prove
\begin{equation}
	\label{eqn:lowfre}
\abs{\int_{\set{t}\times S^3} \partial_t u} + \sum_{i=1}^4 \abs{\int_{\set{t}\times S^3}\partial_t u\cdot \omega_i}\leq C \norm{v}_{\mathcal X_3}
\end{equation}
for $\set{t}\times S^3=\Omega_L,\Omega_R$.
Since we have assumed that the left hand side of \eqref{eqn:lowfre} is zero when $t=\frac{1}{2}\log \lambda$, it suffices to study the change of it as a function of $t$.

To estimate $\int_{\set{t}\times S^3} \partial_t u$, we use the divergence theorm to see that
\begin{equation}
	\abs{\int_{\Omega_L} \partial_t u} =\abs{\int_{\Omega_L} \partial_t u - \int_{\Omega_C}\partial_t u} \leq \int_{[\log \delta, \frac{1}{2}\log \lambda]\times S^3} \abs{v} \leq C \norm{v}_{\mathcal X_3}.	
	\label{eqn:prooflow}
\end{equation}
The same argument works for $\Omega_R$ as well.

To estimate the change of $\int_{\set{t}\times S^3} \partial_t u\cdot \omega_i$, we set
\[
	g_i(t):=\int_{ \set{t}\times S^3} \partial_t u \cdot \omega_i;\quad h_i(t):=\int_{ \set{t}\times S^3} u\cdot \omega_i.
\]
By the fact that $\triangle_{S^3}\omega_i = -3 \omega_i$, we integrate by parts to see
\begin{equation}
	\label{eqn:ghp}
	h_i''(t)=g_i'(t)= \int_{ \set{t}\times S^3} (v-\triangle_{S^3} u) \cdot \omega_i = \int_{ \set{t}\times S^3} v\cdot \omega_i +3 h_i(t).
\end{equation}
If we set
\[
	v_i(t)= \int_{ \set{t}\times S^3} v\cdot \omega_i,
\]
the equation \eqref{eqn:ghp} becomes an ODE of $h_i$
\[
	h_i''(t)-3h_i(t)=v_i(t).
\]
For $t=\frac{1}{2}\log \lambda$, we know
\begin{eqnarray*}
	h_i'(\frac{1}{2}\log \lambda)&=&0\\
	\abs{h_i(\frac{1}{2}\log \lambda)}&\leq& C\lambda^{\alpha_2/2},
\end{eqnarray*}
where we have used \eqref{eqn:ghp} and the assumption that $g_i'(\frac{1}{2}\log \lambda)=\bar{b}_i=0$.
The above initial value problem of ODE has an explicit formula for its solution, which we combine with the decay assumption of $v_i$ to conclude
\[
	\abs{h_i(t)}\leq C \norm{v}_{\mathcal X_3} \eta(t)^{\alpha_2}.
\]
\begin{rem}
	Here it is important to have $\alpha_2>\sqrt{3}$.
\end{rem}
By \eqref{eqn:ghp}, we obtain the same upper bound for $g'_i(t)$. Since $g(\frac{1}{2}\log \lambda)=0$, we obtain by integration
\[
\abs{g_i}\leq C\norm{v}_{\mathcal X_3} \eta(t)^{\alpha_2}.
\]
The proof of \eqref{eqn:lowfre} is done by taking $t=\log \delta$ and $t=\log \lambda-\log \delta$.

{\bf Step 4.} Given the inequality \eqref{eqn:boundaryut}, we may assume that $\partial_t u|_{\Omega_R\cup \Omega_L}$ is zero. To see this, let $u_L$ be the harmonic function on 
	\[
		[\log \delta - 3, \log \delta] \times S^3
	\]
	satisfying the boundary condition
	\[
		\partial_t u_L|_{\Omega_L}=\partial_t u|_{\Omega_L}; \quad u_L|_{ \set{\log \delta -3}\times S^3}=0.
	\]
	Let $\varphi_L(t)$ be the cut-off function that is identically $1$ on $[\log \delta -2, \log\delta]$ and vanishes near $t=\log \delta-3$. By the elliptic estimate, there is universal constant $C$ such that
	\[
		\norm{\varphi_L u_L}_{C^{\alpha_1}([\log \delta -3,\log\delta]\times S^3)} \leq C \norm{\partial_t u}_{C^{\alpha_2}(\Omega_L)},
	\]
	which implies that
\[
	\norm{\triangle(\varphi_L u_L)}_{\mathcal X_3} \leq C \norm{\partial_t u}_{C^{\alpha_2}(\Omega_L)}.
\]
The same argument gives $\varphi_R$ and $u_R$ such that
\[
 \norm{\varphi_R u_R}_{C^{\alpha_1}([\log\lambda-\log\delta,\log\lambda-\log\delta+3]\times S^3)}\leq C \norm{\partial_t u}_{C^{\alpha_2}(\Omega_R)}
\]
and
\[
	\norm{\triangle(\varphi_R u_R)}_{\mathcal X_3} \leq C \norm{\partial_t u}_{C^{\alpha_2}(\Omega_R)}.
\]
Setting $\underline{u}=u-\varphi_L u_L - \varphi_R u_R$, due to \eqref{eqn:boundaryut}, the proof of the lemma is then reduced to the claim that
\begin{equation}
	\norm{\underline{u}}_{\mathcal X_1} \leq C \norm{v-\triangle(\varphi_L u_L + \varphi_R u_R)}_{\mathcal X_3}.
\end{equation}
Notice that by the construction of $\underline{u}$, for the proof of this lemma, we may assume (in addition to Step 1 and Step 2) that $\partial_t u$ vanishes on $\partial \Omega$.

{\bf Step 5.}  Let $u$ be as given in the lemma with the additional assumptions 
\begin{itemize}
	\item $b_0=\bar{a}_i=\bar{b}_i=0$ in \eqref{eqn:bab} (by Step 2);
	\item $\partial_t u|_{\Omega_L\cup \Omega_R}=0$ (by Step 4);
	\item $\triangle u\equiv 0$ on $[\log \lambda -\log \delta+5, \log \delta-5]\times S^3$ (by Step 1).
\end{itemize}
In this step, we establish a global estimate
\[
	\int_{\Omega} u^2 \leq C \norm{v}_{\mathcal X_3}^2.
\]
Multiplying both sides of the equation by $u$ and integrating by parts, we obtain (by the assumption that $\partial_t u|_{\partial\Omega}=0$)
\begin{equation}
	\label{eqn:uvu}
	\int_{\Omega} \abs{\nabla u}^2 = - \int_{\Omega} v\cdot u.
\end{equation}
Setting $g(t)=\int_{\set{t}\times S^3} u$, we apply the Poincar\'e inequality to see
\begin{eqnarray*}
	\int_{\Omega} \abs{u-g(t)}^2 &=& \int_{[\log\lambda-\log\delta,\log\delta]} \int_{S^3} \abs{u-g(t)}^2 \\
	&\leq& \int_{[\log\lambda-\log\delta,\log\delta]} \int_{S^3} \abs{\nabla_{S^3} u}^2 \\
	&\leq& \int_{\Omega} \abs{\nabla u}^2 \\
	&\stackrel{\text{\ref{eqn:uvu}}}{\leq}& \norm{v}_{L^2(\Omega)} \norm{u}_{L^2(\Omega)} \\
	&\leq& \norm{v}_{L^2(\Omega)} \norm{u-g(t)}_{L^2(\Omega)} + \norm{v}_{L^2(\Omega)} \norm{g(t)}_{L^2(\Omega)}.
\end{eqnarray*}
By Young's inequality, we have 
\[
\norm{u-g(t)}_{L^2}\leq C\left( \norm{v}_{L^2(\Omega)} + \norm{g}_{L^2(\Omega)} \right).
\]
For an estimate of $g(t)$, we notice that $g(\frac{1}{2}\log \lambda)=0$ and that we have proved in \eqref{eqn:prooflow}
\[
	\abs{g'(t)}\leq C \norm{v}_{\mathcal X_3} \eta(t)^{\alpha_2}.
\]
Hence,
\[
	\abs{g(t)}\leq C \norm{v}_{\mathcal X_3} \eta(t)^{\alpha_2}.
\]
This implies that $\norm{g}_{L^2(\Omega)}\leq C\norm{v}_{\mathcal X_3}$ (with a constant independent of the length of the cylinder). Therefore, we obtain
\begin{equation}
	\int_{\Omega} u^2  \leq C\norm{v}_{\mathcal X_3}^2.
	\label{eqn:u2}
\end{equation}

{\bf Step 6.}
By \eqref{eqn:u2}, the elliptic estimate near the boundary implies that
\begin{equation}
	\norm{u}_{C^{\alpha_1}([\log \lambda -\log\delta, \log\lambda-\log \delta+6])}, \norm{u}_{C^{\alpha_1}([\log \delta -6, \log \delta])}\leq C\norm{v}_{\mathcal X_3}.
	\label{eqn:goodboundary}
\end{equation}
Over the cylinder $[\log \lambda -\log \delta +5, \log\delta -5]\times S^3$, $u$ is a harmonic function with bounded $L^2$ norm. Hence it has an expansion
\begin{equation}
	u=  a_0+b_0 t + \sum_{i=1}^4 (a_i e^{\sqrt{3}t} +b_i e^{-\sqrt{3}(t-\log\lambda)})\omega_i + (\text{higher order terms.})
	\label{eqn:gooda}
\end{equation}
However, the assumption made in Step 2 together with $\int_{\Omega_C}u=0$ imply that for $i=0,1,\cdots,4$, $a_i=b_i=0.$
Since $\alpha_2>\sqrt{3}$ and the next term in the expansion is of order $\sqrt{8}$, we have (for $t\in [\log \lambda -\log \delta+6,\log \delta -7]$)
\begin{equation}
	\norm{u}_{C^{\alpha_1}(\Omega_{[t]})}\leq C \norm{u}_{L^2(\Omega)} \eta(t)^{\alpha_2}\leq C \norm{v}_{\mathcal X_3} \eta(t)^{\alpha_2}.
	\label{eqn:goodcenter}
\end{equation}
Combining \eqref{eqn:goodcenter} and \eqref{eqn:goodboundary}, we get the estimate
\[
	\norm{u}_{\mathcal X_1}\leq C\norm{v}_{\mathcal X_3}
\]
which concludes the proof of Lemma \ref{lem:neumann}.
\end{proof}

\begin{lem}
	\label{lem:dirichlet}
	Suppose that $u\in C^{\alpha_1}(\Omega)$ satisfying $\triangle u=v$ and the boundary conditions
	\begin{equation}
		\label{eqn:assumedb}
	\Psi(u|_{\Omega_R})= \Psi(u|_{\Omega_L})=0.
	\end{equation}
	If 
	\begin{equation}
		\label{eqn:centeru}
	\int_{\Omega_C} u = \int_{\Omega_C} \partial_t u= \int_{\Omega_C} u\cdot \omega_i = \int_{\Omega_C} \partial_t u\cdot \omega_i=0,
	\end{equation}
then
\begin{equation}
	\label{eqn:dirichlet}
	\norm{u}_{\mathcal X_1}\leq C \norm{v}_{\mathcal X_3}.
\end{equation}
\end{lem}

The version that we state and prove there involves only homogeneous boundary conditions and its proof is easier than that of Lemma \ref{lem:neumann}. We give a brief outline.
\begin{proof}
	{\bf Step 1.} We claim that it suffices to prove the lemma for $v$ that is identically $0$ for $t\in [\log \lambda-\log\delta +5, \log\delta-5]$. In fact, by Theorem \ref{thm:linear} and the elliptic estimates, we find $\tilde{u}$ defined on $\Omega$ satisfying
	\begin{equation*}
		\triangle\tilde{u}=v
	\end{equation*}
and		
	\begin{equation}
	\label{eqn:smallcenteru}
		\norm{ \tilde{u}}_{C^{\alpha_1}(\Omega_{[t]})}\leq \norm{v}_{\mathcal X_3}\eta^{\alpha_2}(t); \quad \forall t\in [\log \lambda -\log\delta +3, \log \delta-3].
	\end{equation}
	Notice that $\tilde{u}$ does not satisfy \eqref{eqn:centeru} in general. However, we may find
	\[
		h= a_0+ b_0(t-\frac{1}{2}\log \lambda) + \sum_{i=1}^4 (a_i e^{\sqrt{3}t} + b_i e^{-\sqrt{3}(t-\log \lambda)} )\omega_i
	\]
	such that $\tilde{u}-h$ satisfies \eqref{eqn:centeru}. Moreover, by \eqref{eqn:centeru} and \eqref{eqn:smallcenteru},
	\[
		\abs{a_0}+\abs{b_0}\leq C\norm{v}_{\mathcal X_3}\lambda^{\alpha_2/2},\quad \abs{a_i}+\abs{b_i}\leq C\norm{v}_{\mathcal X_3}\lambda^{(\alpha_2-\sqrt{3})/2}.
	\]
	Therefore $\norm{h}_{\mathcal X_1}\leq C \norm{v}_{\mathcal X_3}$. 

	Let $\varphi$ be the cutoff function used in Step 1 of the proof of Lemma \ref{lem:neumann}. We verify directly that $u-\varphi(\tilde{u}-h)$ (as $u$) satisfies the extra assumption in the above claim in addition to all assumptions in the lemma. Moreover,
	\begin{eqnarray*}
		\norm{\triangle (u-\varphi(\tilde{u}-h))}_{\mathcal X_3} &\leq& \norm{v}_{\mathcal X_3} + \norm{\varphi(\tilde{u}-h)}_{\mathcal X_1}\\
									 &\leq& C \norm{v}_{\mathcal X_3}.
	\end{eqnarray*}
	If the lemma holds with the extra assumption, i.e. 
	\[
		\norm{u-\varphi(\tilde{u}-h)}_{\mathcal X_1}\leq C \norm{\triangle(u-\varphi(\tilde{u}-h))}_{\mathcal X_3},
	\]
	then we have
	\[
		\norm{u}_{\mathcal X_1}\leq C \norm{v}_{\mathcal X_3}.
	\]

	{\bf Step 2.} Set
	\[
		g(t)=\int_{\set{t}\times S^3} u; \quad g_i(t)= \int_{\set{t}\times S^3}  u\cdot \omega_i.
	\]
	As in the Step 3 in Lemma \ref{lem:neumann}, we get
	\[
		\abs{g(t)}+ \abs{g_i(t)}\leq C \norm{v}_{\mathcal X_3} \eta^{\alpha_2}(t).
	\]
	By \eqref{eqn:unif} and \eqref{eqn:assumedb}, we obtain
	\begin{equation*}
	\norm{u|_{\Omega_L}}_{C^{\alpha_2}(S^3)}, \norm{u|_{\Omega_R}}_{C^{\alpha_2}(S^3)}\leq C \norm{v}_{\mathcal X_3}.
	\end{equation*}

	{\bf Step 3.} The same argument in the Step 4 of Lemma \ref{lem:neumann} allows us to reduce the proof of the lemma to the case
	\[
		u|_{\Omega_R}=u|_{\Omega_L}=0.
	\]
	Using this zero Dirichlet boundary value, we do integration by parts to get (as in Step 5 there)
	\[
		\int_{\Omega} u^2 \leq C \norm{v}_{\mathcal X_3}^2.
	\]

	{\bf Step 3.} The final step is the same as the Step 6 in Lemma \ref{lem:neumann}.
\end{proof}

\section{Computation on Cylinder}
In this section, we collect some elementary computations.

\subsection{Proof of \eqref{eqn:directcom}}
Let $e_1,e_2,e_3$ be an orthonomal basis of one forms on $S^3$. Assume that the positive orientation of the cylinder $\Real\times S^3$ is given by the frame $(dt, e_1,e_2,e_3)$. Hence, we may compute
\begin{eqnarray*}
	\triangle_h (f dt) &=& - (*d*d + d*d*)(fdt)\\
	&=& - * d * (d_{S^3}f \wedge dt) - d*d (f e_1\wedge e_2\wedge e_3) \\
	&=&* d * (dt \wedge d_{S^3} f) - d* (\partial_t f dt\wedge e_1\wedge e_2\wedge e_3)\\
	&=&* d (*_{S^3} d_{S^3} f) - d(\partial_t f) \\
	&=&* dt\wedge *_{S^3} d_{S^3}\partial_t f + * d_{S^3} *_{S^3} d_{S^3} f - \partial_t^2 f dt - d_{S^3} \partial_t f \\
	&=& *_{S^3}^2 d_{S^3} \partial_t f - (*_{S^3} d_{S^3} *_{S^3} d_{S^3}f) dt - \partial_t^2 f dt - d_{S^3}\partial_t f \\
	&=& (-\triangle_{S^3} f -\partial_t f) dt.
\end{eqnarray*}
Notice that here in the last line $\triangle_{S^3}$ is the Laplacian for analysist, while the Hodge Laplacian for function $f$ on $S^3$ is given by $-*_{S^3} d_{S^3} *_{S^3}d_{S^3}$.

For the other half, we compute
\begin{eqnarray*}
	\triangle_h \xi(x,t) &=& - * d * d \xi - d * d* \xi \\
	&=& -*d* dt\wedge(\partial_t \xi) - * d* d_{S^3}\xi + d* d (*_{S^3}\xi \wedge dt)\\
	&=&-*d*_{S^3}(\partial_t \xi) - * d \left( dt\wedge *_{S^3}d_{S^3}\xi \right) + d*\left( d_{S^3}*_{S^3}\xi \wedge dt \right) \\
	&=&-* dt\wedge(*_{S^3} \partial_t^2 \xi) - *d_{S^3}*_{S^3}\partial_t \xi + * dt\wedge d_{S^3} *_{S^3} d_{S^3}\xi - d(*_{S^3} d_{S^3} *_{S^3} \xi)\\
	&=& - *_{S^3}^2 \partial_t^2 \xi + dt\wedge (*_{S^3}d_{S^3}*_{S^3} \partial_t \xi) + *_{S_3}d_{S^3}*_{S^3}d_{S^3}\xi - d_{S^3}*_{S^3}d_{S^3}*_{S_3} \xi - dt\wedge \partial_t(*_{S^3} d_{S^3} *_{S^3}\xi)\\
	&=& -\partial_t^2 \xi + \triangle_{h,S^3}\xi.
\end{eqnarray*}

\subsection{Proof of \eqref{eqn:collect}}
If we write $\phi$ for one of $\phi_{\pm,i}$, we have
\begin{eqnarray*}
	d^*(e^{2t}\phi)&=&-*d*(e^{2t}\phi)\\
	&=&-*d(e^{2t}dt\wedge *_{S^3}\phi)\\
	&=& *\left( e^{2t}dt\wedge d_{S^3}*_{S^3}\phi \right)\\
	&=& e^{2t} *_{S^3}d_{S^3}*_{S^3} \phi\\
	&=&e^{2t} d_{S^3}^* \phi=0,
\end{eqnarray*}
since $\phi_{\pm,i}$ is coclosed.

Recall that $\psi_i$ is defined to be $d_{S^3}\omega_i$ if $\omega_i$ is regarded as a function on cylinder independent of $t$.
Hence, we compute
\begin{eqnarray*}
	d^*(e^{\sqrt{3}t} \psi_i) &=& - * d * (e^{\sqrt{3}t} d_{S^3}\omega_i)\\
				  &=& * d \left( e^{\sqrt{3}t} dt \wedge  *_{S^3}d_{S^3} \omega_i \right) \\
				  &=& - * e^{\sqrt{3}t} dt\wedge d_{S^3}*_{S^3}d_{S^3}\omega_i \\
				  &=& - e^{\sqrt{3}t} *_{S^3}d_{S^3}*_{S^3}d_{S^3}\omega_i \\
				  &=& -e^{\sqrt{3}t} \triangle_{S^3} \omega_i \\
				  &=& 3 e^{\sqrt{3}t}\omega_i.
\end{eqnarray*}
Notice that for a function on $S^3$, the Hodge Laplace $\triangle_{h;S^3}= \delta d = -*_{S^3}d_{S^3}*_{S^3}$ and our convention is that $\triangle_{S^3}$ is minus the Hodge Laplace.

The proof for the last line in \eqref{eqn:collect} is the same.

\bibliographystyle{alpha}
\bibliography{foo}
\end{document}